\documentclass[12pt,reqno]{amsart}
\usepackage{amsmath,amssymb,amscd,graphicx}
\usepackage[all,2cell]{xy}
\usepackage{ifthen}
\usepackage{subsupscripts}
\usepackage[margin=1.25in]{geometry}	
\usepackage{url}
\usepackage{cite}
\usepackage{tikz}
\usepackage{hyperref}

\newboolean{workmode}
\setboolean{workmode}{false}

\newboolean{sections}
\setboolean{sections}{true}

\UseComputerModernTips
\UseTwocells

\def\eoe{\unskip\ \hglue0mm\hfill$\diamond$\smallskip\goodbreak}

\makeatletter
\def\th@plain{%
  \thm@notefont{}
  \itshape 
}
\def\th@definition{%
  \thm@notefont{}
  \normalfont 
}
\makeatother

\newcommand{\calA}{\mathcal{A}}

\newcommand{\calE}{\mathcal{E}}
\newcommand{\calF}{\mathcal{F}}
\newcommand{\calG}{\mathcal{G}}
\newcommand{\calH}{\mathcal{H}}

\newcommand{\CC}{{\mathbb{C}}}  
\newcommand{\KK}{{\mathbb{K}}}
\newcommand{\RR}{{\mathbb{R}}}  
\newcommand{\ZZ}{{\mathbb{Z}}}  

\newcommand{\Ad}{{\operatorname{Ad}}}  
\newcommand{\id}{{\operatorname{id}}}  
\newcommand{\pr}{{\operatorname{pr}}} 

\newcommand{\G}{{\operatorname{G}}} 
\newcommand{\g}{{\mathfrak{g}}} 
\DeclareMathOperator{\PU}{PU}  
\newcommand{\U}{{\operatorname{U}}}  

\newcommand{\GG}{\mathbf{G}}  
\newcommand{\toto}{{~\rightrightarrows~}} 
\makeatletter						
\newtheorem*{rep@theorem}{\rep@title}
\newcommand{\newreptheorem}[2]{
\newenvironment{rep#1}[1]{
\def\rep@title{#2 ##1}
\begin{rep@theorem}}
{\end{rep@theorem}}}
\makeatother

\newcommand{\ifsection}[2]{\ifthenelse{\boolean{sections}}{#1}{#2}}

\theoremstyle{plain}
\ifsection{
    \newtheorem{theorem}{Theorem}[section]
    
	\numberwithin{equation}{section}
	\numberwithin{figure}{section}
	\usepackage{chngcntr}
	\counterwithin{table}{section}
}
{
    \newtheorem{theorem}{Theorem}
}

\newreptheorem{theorem}{Theorem}		
\newreptheorem{proposition}{Proposition}
\newreptheorem{corollary}{Corollary}
\newtheorem{proposition}[theorem]{Proposition}

\newtheorem{corollary}[theorem]{Corollary}
\newtheorem{lemma}[theorem]{Lemma}

\theoremstyle{definition}
\newtheorem{definition}[theorem]{Definition}
\newtheorem{example}[theorem]{Example}
\newtheorem{remark}[theorem]{Remark}

\newcommand{\calBBS}{{\mathcal{B}^2\mathcal{S}^1}} 
\newcommand{\opDC}{{\operatorname{DC}}} 
\newcommand{\DC}{\mathcal{DC}} 
\newcommand{\dashto}{{\;\dashrightarrow\;}} 

\newcommand{\PPhi}{\mathbf{\Phi}}

\author{Derek Krepski}
\address{
Department of Mathematics, University of Manitoba, 
Winnipeg, MB, Canada 
}
\email{\href{mailto:Derek.Krepski@umanitoba.ca}{Derek.Krepski@umanitoba.ca}}
\urladdr{\href{http://server.math.umanitoba.ca/~dkrepski/}{http://server.math.umanitoba.ca/\textasciitilde dkrepski/}}

\title{Groupoid equivariant prequantization}
\date{\today}

\thanks{This work is partially supported by an NSERC Discovery Grant.}


\begin{document}

\begin{abstract}
In their 2005 paper, C.~Laurent-Gengoux and P.~Xu  define  prequantization for pre-Hamiltonian actions of quasi-presymplectic Lie groupoids in terms of $S^1$-central extensions of Lie groupoids.  The definition requires that the quasi-presymplectic structure be \emph{exact} (i.e.~the closed 3-form on the unit space of the Lie groupoid must be exact).  In the present paper, we define prequantization for pre-Hamiltonian actions of (not necessarily exact) quasi-presymplectic Lie groupoids in terms of Dixmier-Douady bundles.  The definition is a natural adaptation of E.~Meinrenken's treatment of prequantization for quasi-Hamiltonian Lie group actions with group-valued moment map. The definition given in this paper is shown to be compatible with the definition of Laurent-Gengoux and Xu when the underlying quasi-preysmplectic structure is exact.  Properties related to Morita invariance and symplectic reduction are established.
\end{abstract}

\maketitle

\section{Introduction}\label{s:intro}

Quasi-symplectic groupoids were introduced in \cite{xu2004momentum} to provide a unified framework to study several `momentum map theories,' 
 and in  \cite{bursztyn2004integration} (where they are called twisted presymplectic groupoids) as the global objects integrating twisted Dirac structures.  Recall that a quasi-\emph{pre}symplectic structure on a Lie groupoid $G_1\toto G_0$ consists of a pair of differential forms
$\eta \in \Omega^3(G_0)$ and $\omega\in \Omega^2(G_1)$ satisfying
\[
d\eta =0,\quad s^*\eta - t^*\eta = d\omega, \quad \text{and} \quad m^*\omega = \pr_1^*\omega+\pr_2^*\omega,
\]
where $s,t:G_1 \to G_0$ are the source and target maps, respectively, $m:G_1\times_{G_0} G_1 \to G_1$ is the groupoid multiplication, and $\pr_1, \pr_2$ are the natural projections.  Examples include (see Section \ref{ss:qsgpd}) the cotangent bundle of a Lie group, the AMM groupoid associated to the conjugation action on a compact Lie group (as well as its twisted conjugation counterpart), and more generally the dressing action groupoid associated to a group pair integrating a Manin pair.  

The resulting momentum map theories consist of Lie groupoid actions of the quasi-symplectic groupoids $G_1\toto G_0$ on manifolds $X$ equipped with a differential form $\omega_X \in \Omega^2(X)$ satisfying a compatibility condition with the quasi-symplectic structure  (see Definition \ref{d:Ham}).  These include ordinary Hamiltonian actions of Lie groups on symplectic manifolds, quasi-Hamiltonian actions with Lie group-valued moment map \cite{alekseev1998lie} and their twisted counterparts \cite{meinrenken2015convexity}, Hamiltonian loop group actions \cite{meinrenken1998hamiltonian}, Poisson-Lie group actions \cite{lu1990poisson,lu1991momentum}, Hamiltonian quasi-Poisson $G$-spaces \cite{alekseev2000manin}, and symmetric-space valued moment maps \cite{leingang2003symmetric}.

In their 2005 paper \cite{laurent2005quantization}, C.~Laurent-Gengoux and P.~Xu define prequantization for pre-Hamiltonian actions of \emph{exact} quasi-presymplectic groupoids---where the 3-form $\eta$ is exact---using central extensions of Lie groupoids (see also \cite{weinstein1991extensions,xu1994classical}).    Consequently, their definition is not immediately applicable in certain examples, notably in the case of quasi-Hamiltonian $G$-actions where $G$ is a compact semi-simple Lie group.  As noted by the authors, one must first pass to a Morita equivalent quasi-presymplectic groupoid (a \v{C}ech groupoid associated to a $G$-equivariant covering of $G$) in order to suitably interpret prequantization in that context. 

Following  E.~Meinrenken's approach in \cite{meinrenken2012twisted} for quasi-Hamiltonian group actions, this paper presents a definition of prequantization (see Definition \ref{d:preq}) for  Hamiltonian actions of quasi-presymplectic groupoids (with no exactness assumption), using the theory of Dixmier-Douady bundles (DD-bundles), which are bundles of $C^*$-algebras with typical fibre the compact operators on an infinite dimensional separable Hilbert space.  An important aspect of this definition relates to the notion of equivariance for DD-bundles used in the present context and investigated further in \cite{KrepskiWatts}, which makes use of the `higher structure' inherent in DD-bundles (that is also present for other models of $S^1$-gerbes, e.g.~see \cite{nikolaus2011equivariance}).

We summarize the main results of the paper.  In Theorem \ref{t:compare}(1), we show that the definition of prequantization presented here (Def.~\ref{d:preq}) is consistent with the definition of Laurent-Gengoux and Xu for pre-Hamiltonian actions of exact quasi-presymplectic Lie groupoids.  In this special case, one obtains a prequantum line bundle over the space $X$ acted upon, and Theorem \ref{t:compare}(2)  gives its (equivariant) real Chern class.  

Recall that Morita equivalent quasi-presymplectic groupoids give rise to a correspondence between their corresponding Hamiltonian spaces.  Theorem \ref{t:moritapreq} verifies a Morita invariance property of prequantization, showing that prequantization respects this correspondence.  As a corollary (Corollary \ref{c:ammloop}), we recover the equivalence of prequantizations for Hamiltonian loop group actions and quasi-Hamiltonian group actions, without the assumption of simple-connectivity on the underlying Lie group in \cite[Theorem A.7]{krepski2008pre}.

Finally, Theorem \ref{t:red}  shows that the prequantization of a pre-Hamiltonian space descends to a prequantization of its symplectic quotients (provided they are smooth---see  Remark \ref{r:quot}).


This paper is organized as follows. 

In Section \ref{s:prelim}, we recall elementary simplicial data related to Lie groupoids and establish the notation used throughout the paper.  We recall the definition of a quasi-presymplectic groupoid from \cite{xu2004momentum} together with some examples appearing in the literature and review the definition of a pre-Hamiltonian action in this setting.

As stated earlier, the definition of prequantization given in this paper uses the theory of DD bundles.  In Section \ref{s:gerbes}, we review some elements of Dixmier-Douady theory (following \cite{alekseev2012dirac}) and recall their groupoid-equivariant counterparts from \cite{KrepskiWatts}, which makes use of the fact that DD bundles from a bicategory.  We also recall the equivalent (strict) 2-category of differential characters of degree 3 (as well as its groupoid-equivariant counterpart), which we make use of in the proofs of several of our results.  We also review the relative versions (i.e.~relative equivariant DD bundles and relative equivariant differential characters) associated to a groupoid morphism.

In Section \ref{s:prequant}, we give the definition of prequantization for pre-Hamiltonian actions in terms of equivariant Dixmier-Douady theory.  After recalling the definition of prequantizion from \cite{laurent2005quantization} for pre-Hamiltonian actions of exact quasi-presymplectic groupoids---phrased in terms of $S^1$-central extensions of Lie groupoids---we state and prove the main result, Theorem \ref{t:compare}, which verifies that the definitions agree in this case.  

In Section \ref{s:morita}, we review the definition of  Morita equivalence for quasi-presymplectic groupoids and recall the correspondence between related Hamiltonian spaces for Morita equivalent quasi-presymplectic groupoids.  The Morita invariance property, stated in Theorem \ref{t:moritapreq} is established.

In the final section, we prove Theorem \ref{t:red} on the compatibility of prequantization with symplectic reduction.

\noindent \emph{Acknowledgements}. The author thanks H.~Bursztyn, E.~Lerman, E.~Meinrenken, and J.~Watts for helpful conversations and correspondence.  This work was partially supported by an NSERC Discovery Grant.

\section{Preliminaries and notation}\label{s:prelim}

In this section we record some preliminaries on Lie groupoids and recall some definitions surrounding quasi-presymplectic groupoids and their pre-Hamiltonian actions.

\subsection{Lie groupoids and double complexes} \label{ss:liegpd}
Mainly to establish notation, we recall aspects related to the simplicial manifold $G_\bullet$ associated to a Lie groupoid $\GG=(G_1 \rightrightarrows G_0)$, as well as the resulting double complex arising from a presheaf of chain complexes.  Denote the source and target maps by $s,t:G_1 \to G_0$, multiplication (composition) by $m:G_1 \times_{G_0} G_1 \to G_1$, inversion by $\iota:G_1 \to G_1$ and the unit by $\epsilon:G_0 \to G_1$.  To avoid the possibility of confusion, we may at times decorate the structure maps with a subscript to indicate the underlying groupoid (e.g. $\epsilon_{\GG}$ as the unit map of $\GG$).

For $k \geq 2$, write 
\[
G_k = \underbrace{G_1 \times_{G_0} G_1 \times_{G_0} \cdots \times_{G_0} G_1}_{k \text{ factors}}
\]
whose elements are $k$-tuples $(g_1, \ldots, g_k)$ of composable arrows (with $s(g_i)=t(g_{i+1})$).  For $0 \leq i \leq k$, let $\partial_i:G_k \to G_{k-1}$ be the \emph{face maps} given by
\[
\partial_i(g_1, \ldots, g_k) = \left\{ 
\begin{array}{ll}
(g_2, \ldots, g_k) & \text{if }i=0 \\
(g_1, \ldots, g_ig_{i+1},\ldots, g_k) & \text{if }0<i<k \\
(g_1, \ldots, g_{k-1}) & \text{if }i=k. \\
\end{array}
\right.
\]
For convenience, we set $\partial_0=s$ and $\partial_1=t$ on $G_1$.  It is easily verified that the face maps satisfy the simplicial identities $\partial_i\partial_j = \partial_{j-1} \partial_i$ for $i<j$.  (We will not require degeneracy maps in this paper.)

A morphism of Lie groupoids $\mathbf{F}: \GG \to \mathbf{H}$ yields a map of simplicial manifolds $F_\bullet:G_\bullet \to H_\bullet$, which for convenience may also be denoted $\mathbf{F}$.  If $X$ is a $\GG$-space with anchor map $\Phi:X\to G_0$, let $\PPhi:\GG\ltimes X \to \GG$  denote the resulting morphism of Lie groupoids.

Let $(C^*,d)$ denote a presheaf of cochain complexes, and consider the double complex $C^*(G_\bullet)$, depicted below.
\[
\xymatrix{
\vdots & \vdots & \vdots \\
C^2(G_0)	\ar[r]^\partial \ar[u]^{d} & C^2(G_1)\ar[r]^\partial \ar[u]^{-d} & C^2(G_2) \ar[r]^\partial \ar[u]^{d}	& \cdots	\\
C^1(G_0)	\ar[r]^\partial \ar[u]^{d} & C^1(G_1)\ar[r]^\partial \ar[u]^{-d} & C^1(G_2) \ar[r]^\partial \ar[u]^{d}	& \cdots	\\
C^0(G_0)	\ar[r]^\partial \ar[u]^{d} & C^0(G_1)\ar[r]^\partial \ar[u]^{-d} & C^0(G_2) \ar[r]^\partial \ar[u]^{d}	& \cdots	\\
}
\]
The horizontal differential is the alternating sum of pullbacks of face maps, $\partial = \sum (-1)^i \partial_i^*$.  Denote the total complex by 
\[
C(\GG):=\mathrm{Tot}(C^*(G_\bullet)), \quad \text{with } C^n(\GG) = \bigoplus_{p+q=n} C^p(G_q),
\]
with    total differential $\delta = (-1)^q d \oplus \partial$.
For example, when $C^*=\Omega^*$ is the de~Rham complex, we obtain the Bott-Shulman-Stasheff complex of $G_\bullet$.  In this paper, we will use the de~Rham complex $\Omega^*$, smooth singular cochains $C^*(-;\ZZ)$ and $C^*(-;\RR)$, and a cochain complex of Hopkins and Singer\cite{hopkins2005quadratic}, denoted $DC^*$ following the notation in \cite{lerman2008differential} (reviewed in Section \ref{ss:dc}).  Note that we will abuse notation and use integration of forms to view $\Omega^*(M)\subset C^*(M;\RR)$ and also view $C^*(M;\ZZ) \subset C^*(M;\RR)$.

\subsection{Quasi-symplectic groupoids and Hamiltonian actions} \label{ss:qsgpd}

We recall some definitions from \cite{xu2004momentum} and \cite{bursztyn2004integration}.

\begin{definition} \label{d:qss}
Let $\GG=(G_1\rightrightarrows G_0)$ be a Lie groupoid.  A \emph{quasi-presymplectic structure} on $\GG$ is a closed differential form $\omega\oplus \eta \in \Omega^2(G_1) \oplus \Omega^3(G_0)$ of degree 3 in the total complex $\Omega^*(\GG)$.
\end{definition}

In terms of the de Rham differential $d$ and simplicial differential $\partial$, the differential forms $\omega$ and $\eta$ must satisfy 
\[
d\eta =0, \quad d\omega=\partial \eta, \quad \text{and}\quad \partial \omega=0.
\]
\begin{remark} As noted in \cite[Remark 2.2]{xu2004momentum}, $\partial\omega=0$ is equivalent to the condition that $\omega$ be \emph{multiplicative}, i.e.~$m^*\omega=\pr_1^*\omega + \pr_2^*\omega$, where $m$ denotes the groupoid multiplication.  Equivalently,  the graph of   $m:G_2 \to G_1$ is an isotropic submanifold of $G_1 \times G_1 \times \bar{G}_1$  (with 2-form $\omega \oplus \omega \oplus (-\omega)$).
\end{remark}
\begin{remark}
Using the terminology in \cite{bursztyn2004integration}, the condition $d\omega=\partial \eta$ is the condition that \emph{$\omega$ is relatively $\eta$-closed}.
\end{remark}

A quasi-presymplectic Lie groupoid $(\GG,\omega\oplus \eta)$ is \emph{quasi-symplectic} if the following  condition controlling the degeneracy of $\omega$ is satisfied \cite{bursztyn2004integration,xu2004momentum}:
\[
\ker(\omega_x) \cap \ker (ds)_x \cap \ker (dt)_x = \{0\} \quad \text{for all } x\in G_0.
\]
In this work, we will not make use of the above condition.

\begin{remark}
Quasi-symplectic groupoids $(\GG,\omega \oplus \eta)$ were also introduced in \cite{bursztyn2004integration}, where they are called called \emph{$\eta$-twisted presymplectic groupoids}.  (In that setting, the prefix \emph{pre} alludes to the fact that $\omega$ is possibly degenerate, and the prefix is removed when $\omega$ is non-degenerate.)
As noted in \cite{xu2004momentum}, symplectic groupoids \cite{weinstein1987symplectic} and twisted symplectic groupoids \cite{cattaneo2004integration} are important special cases of quasi-symplectic groupoids.  
\end{remark}

We recall some familiar examples below.

\begin{example}[Cotangent bundle of a Lie group] \label{eg:TG} Let $G$ be a Lie group with Lie algebra $\g$.
The symplectic groupoid  $T^*G \rightrightarrows \mathfrak{g}^*$ (with canonical symplectic form on $T^*G$ and zero form on $\g^*$) is a quasi-symplectic groupoid. 
\eoe  
\end{example}

\begin{example}[AMM groupoid] \label{eg:amm} \cite{alekseev1998lie,bursztyn2004integration,xu2004momentum}
Let $G$ be a Lie group equipped with a bi-invariant inner product $\langle\cdot, \cdot \rangle$ on $\g$.  Let $G\times G \rightrightarrows G$ be the action groupoid for the conjugation action of $G$ on itself.  Let $\theta^L$ and $\theta^R$ denote, respectively, the left and right invariant Maurer-Cartan forms on $G$. Let $\eta \in \Omega^3(G)$ be the Cartan form,
\[
\eta =\frac{1}{12} \langle \theta^L,[\theta^L,\theta^L] \rangle,
\]
and $\omega \in \Omega^2(G\times G)$ be given by
\[
\omega_{(g,x)} = -\frac{1}{2} \left( \langle \Ad_x \pr_1^*\theta^L,\pr_1^*\theta^L \rangle + \langle \pr_1^*\theta^L,\pr_2^*(\theta^L+\theta^R) \rangle
 \right)
\]
Then $(G\ltimes G \rightrightarrows G,\omega \oplus \eta)$ is a quasi-symplectic groupoid.
\eoe
\end{example}

\begin{example}[Twisted AMM groupoid] \label{eg:ammtwist}
The AMM groupoid above may be twisted by an automorphism of the Lie group $G$, as in \cite{meinrenken2015convexity}.  Let $\kappa:G\to G$ be an automorphism, and consider the \emph{$\kappa$-twisted conjugation action} of $G$ on itself
\[
\Ad_g^{(\kappa)} (x) = gx\kappa(g^{-1}).
\]
Let $G\kappa$ denote the manifold $G$ equipped with the above action.
Let $\eta$ be the Cartan form as above, and let $\omega^{(\kappa)} \in \Omega^2(G\times G\kappa)$ be given by
\[
\omega^{(\kappa)}_{(g,x)}= -\frac{1}{2} \left( \langle \Ad_x \pr_1^* \kappa^*\theta^L, \pr_1^*\theta^L \rangle  +\langle \pr_1^*\kappa^*\theta^L,\pr_2^*\theta^L\rangle + \langle \pr_1^*\theta^L,  \pr_2^*\theta^R\rangle \right).
\]
As in the untwisted case, it is straightforward to verify that $(G\ltimes G\kappa \rightrightarrows G\kappa,\omega^{(\kappa)}\oplus \eta)$ is a quasi-symplectic groupoid, since  the inner product is invariant under automorphisms $\kappa$.
\eoe
\end{example}

\begin{example}[Manin pairs]
Let $\mathfrak{d}$ be a Lie algebra equipped with an invariant, nondegenerate, symmetric bilinear form of split signature and $\g\subset \mathfrak{d}$ a Lie subalgebra that is also a maximal isotropic subspace.   Suppose that the pair $(\mathfrak{d},\mathfrak{g})$ (the \emph{Manin pair} \cite{alekseev2000manin}) integrates to a \emph{group pair} $(D,G)$ consisting of a connected Lie group $D$ with Lie algebra $\mathfrak{d}$, and $G\subset D$ a connected closed Lie subgroup with Lie algebra $\g$. Let $D/G$ denote the space of right cosets, and consider the \emph{dressing} action of $G\subset D$ on $D/G$ induced by left translation.
As shown in \cite{bursztyn2009dirac}, a choice of \emph{isotropic} connection on the principal $G$-bundle $D\to D/G$ determines a quasi-presymplectic structure on $G\ltimes D/G \toto D/G$.    

Examples \ref{eg:TG}, \ref{eg:amm}, and \ref{eg:ammtwist} may all be obtained via Manin pairs.  Another important example of this kind is the action groupoid $G\ltimes G^* \toto G^*$ of a complete Poisson-Lie group $G$  acting on its Poisson-Lie dual group $G^*$ by the left dressing action \cite{lu1990poisson,lu1991momentum}.
\eoe
\end{example}

\begin{example}[Coadjoint action groupoid of loop groups] \label{eg:loopgroup} \cite{presley1986loop,behrend2003equivariant}
Let $G$ be a compact Lie group with bi-invariant inner product on its Lie algebra $\g$.  
Fix a real number $s>1$, and consider the loop group $LG=\operatorname{Map}(S^1,G)$ consisting of maps of Sobolev class $s+1/2$ .
The inner product on $\g$ induces a 2-cocyle on the loop Lie algebra $L\g$ and hence corresponds to a central extension of Lie algebras $\widehat{L\g}=L\g \oplus \RR$.  Suppose the central extension integrates to a central extension of Lie groups
\[
1\to S^1 \to \widehat{LG} \to LG \to 1.
\]
Let $L\g^*=\Omega^1(S^1;\g)$ be the space of $\g$-valued 1-forms of Sobolev class $s-1/2$, with pairing $L\g\times L\g^* \to \RR$ given by $(\xi,A) \mapsto \int_{S^1} \langle \xi,A \rangle$.  Let $\widehat{L\g^*}=L\g^*\oplus \RR$, and observe that the coadjoint action of $\widehat{LG}$ on $\widehat{L\g^*}$ factors through $LG$.  Identifying $L\g^* \cong L\g^*\oplus \{1\}\subset \widehat{L\g^*}$  recovers the standard $LG$-action on $L\g$ by gauge transformations.  Let $\omega_{\mathsf{can}}$ be the 2-form on $\widehat{LG}\times \widehat{L\g^*}$, defined by the same formula as the canonical symplectic form on the trivialized cotangent bundle $T^*\widehat{LG}$.  Then the restriction of $\omega_{\mathsf{can}}$ to $\widehat{LG} \times L\g^*$ descends to a 2-form $\nu$ on $LG \times L\g^*$, and $(LG\ltimes L\g^* \toto L\g^*,\nu)$ is a quasi-symplectic groupoid. 
\eoe
\end{example}

\begin{definition} \label{d:Ham}
Let $\GG = (G_1 \rightrightarrows G_0, \omega \oplus \eta)$ be a quasi-presymplectic groupoid.  
A \emph{pre-Hamiltonian $\GG$-space} is a triple $(X,\omega_X,\Phi)$ consisting of  a (left) $\GG$-action on manifold  $X$ with anchor $\Phi:X \to G_0$ together with a  2-form $\omega_X$ on $X$ satisfying
\[
\PPhi^*(\omega\oplus\eta) = -\delta \omega_X \in \Omega^*(\GG \ltimes X)
\]
where $\GG \ltimes X$ denotes the action groupoid for the $\GG$-action on $X$.
\end{definition}
\begin{remark}
Note that Definition \ref{d:Ham} employs the opposite sign convention for pre-Hamiltonian spaces appearing in \cite{xu2004momentum}.
\end{remark}

Hamiltonian actions by quasi-symplectic groupoids were introduced in \cite{xu2004momentum} as pre-Hamiltonian actions satisfying a \emph{minimal degeneracy condition}, controlling the degeneracy of the 2-form $\omega_X$.  This paper does not make use of the minimal degeneracy condition.

Applying Definition \ref{d:Ham} to the examples of quasi-presymplectic groupoids listed above provides several important examples of (pre)-Hamiltonian actions, such as:
ordinary Hamiltonian $G$-actions on symplectic manifolds, quasi-Hamiltonian $G$-actions with group-valued moment map \cite{alekseev1998lie} and their twisted counterparts \cite{meinrenken2015convexity}, Hamiltonian loop group actions \cite{meinrenken1998hamiltonian}, Poisson-Lie group actions \cite{lu1990poisson,lu1991momentum}, Hamiltonian quasi-Poisson $G$-spaces \cite{alekseev2000manin}, and symmetric-space valued moment maps \cite{leingang2003symmetric}.

\section{Equivariant  Dixmier-Douady bundles and differential characters} \label{s:gerbes}

In this section, we recall some perspectives on  Dixmier-Douady bundles and differential characters.  
We also briefly review their groupoid equivariant counterparts as in \cite{KrepskiWatts}.  
Dixmier-Douady bundles are geometric models for \emph{$S^1$-gerbes}.  Other such models appearing in the literature are $S^1$-central extensions (see Definition \ref{d:s1centext}), bundle gerbes \cite{murray1996bundle}, and principal Lie 2-group bundles \cite{baez2007higher}. 

\subsection{Dixmier-Douady bundles} \label{ss:dd}
We provide a brief review of Dixmier-Douady bundles, following \cite{alekseev2012dirac}.  For further background, see also \cite{raeburn1998morita}.

Let $\calH$ denote an infinite dimensional separable Hilbert space, and $\KK(\calH)$ the $C^*$-algebra of compact operators on $\calH$.  
Recall that the automorphism group $\operatorname{Aut}(\KK(\calH))\cong \PU(\calH)=\U(\calH)/S^1$, where $\U(\calH)$ (with strong operator topology) acts on $\KK(\calH)$ by conjugation. 

Recall that a \emph{Dixmier-Douady bundle} (\emph{DD-bundle}) $\calA\to M$ is a locally trivial bundle of $C^*$-algebras with typical fibre $\KK(\calH)$ and structure group $\PU(\calH)$.  

A \emph{Morita isomorphism} of DD-bundles $\calE\colon \calA_1\dashto \calA_2$ is a locally trivial Banach space bundle $\calE\to M$ of $\calA_2-\calA_1$ bimodules with typical fibre  $\KK(\calH_1,\calH_2)$, the compact operators from $\calH_1$ to $\calH_2$.  Locally, the bimodule action is given fibrewise by the natural $\KK(\calH_2)-\KK(\calH_1)$ bimodule action given by post- and pre-composition of operators, respectively.
The composition of two Morita isomorphisms $\calE_1\colon\calA_1\dashto\calA_2$ and $\calE_2\colon\calA_2\dashto\calA_3$ is given by $\calE_2\circ\calE_1=\calE_2\otimes_{\calA_2}\calE_1$,  the fibrewise completion of the (algebraic) tensor product over $\calA_2$.

Given two Morita isomorphisms $\calE_1,\calE_2\colon\calA_1\dashto\calA_2$, a \emph{$2$-isomorphism} $\tau\colon\calE_1\Rightarrow\calE_2$ is a continuous bundle isomorphism $\tau:\calE_1\to\calE_2$ that intertwines the norms and the ($\calA_2-\calA_1$)-bimodule structures.

Recall that given a Morita isomorphism $\calE\colon \calA_1 \dashto \calA_2$, the \emph{opposite} Morita isomorphism $\calE^*\colon \calA_2 \dashto \calA_1$ is given by $\calE^*=\calE$ as real vector bundles, with opposite (conjugate) scalar multiplication.  There are natural 2-isomorphisms $\calE^* \otimes_{\calA_2} \calE \cong \calA_1$ and $\calE \otimes_{\calA_1} \calE^* \cong \calA_2$.  

\begin{remark} \label{r:reorder}
Suppose given Morita isomorphisms $\calE_1$, $\calE_2$ and $\calF$ as well as a 2-isomorphism $\tau\colon \calE_2 \otimes_{\calA_2} \calE_1 \to \calF$ as in the diagram below.
\[
\xymatrix{
\rrtwocell<\omit>{<5>\tau} & \calA_2 \ar@{-->}[dr]^{\calE_2} & \\
 \calA_1 \ar@{-->}[ur]^{\calE_1} \ar@{-->}[rr]_{\calF} & & \calA_3
}
\]
At times, we will abuse notation and view $\tau$ instead as the composition 
\[
\xymatrix{ 
\calE_1 \ar[r]^-\cong & {\calA_3 \otimes_{\calA_3} \calE_1} \ar[r]^-\cong & {\calE_2^* \otimes_{\calA_3} \calE_2 \otimes_{\calA_2} \calE_1} \ar[r]^-{\id\otimes \tau} & {\calE_2^* \otimes_{\calA_3} \calF}
}.
\]
\end{remark}

The above definitions allow us to view DD-bundles over a fixed manifold $M$ as a bigroupoid (i.e.~weak 2-category with coherently invertible 1-arrows and invertible 2-arrows).  Hence, for example, we may speak of `horizontal' composition  of 2-isomorphisms $\tau \otimes \sigma:\calF_1 \otimes_{\calA_2} \calE_1 \Rightarrow \calF_2\otimes_{\calA_2} \calE_2$:
\[
\xymatrix{
\calA_1 \ar@{-->}@/^1pc/[r]^{\calE_1} \ar@{-->}@/_1pc/[r]_{\calE_2}\xtwocell[0,1]{}\omit{\sigma}  & \calA_2 \ar@{-->}@/^1pc/[r]^{\calF_1} \ar@{-->}@/_1pc/[r]_{\calF_2}\xtwocell[0,1]{}\omit{\tau}  & \calA_3
}
\]
as well as `vertical' composition $\tau\circ \sigma$ of 2-isomorphisms, $\calE \stackrel{\sigma}{\Longrightarrow} \calF  \stackrel{\tau}{\Longrightarrow} \calG$,
\[
\xymatrix{
\calA_1 \ar@{-->}@/^3pc/[r]^{\calE} \ar@{-->}@/_3pc/[r]_{\calG} \xtwocell[0,1]{}\omit{<-4>\sigma} \xtwocell[0,1]{}\omit{<4>\tau} \ar@{-->}[r]^{\calF}& \calA_2
}
\]
which is the usual composition of the underlying bundle maps.



\subsubsection{Equivariant Dixmier-Douady bundles} \label{ss:eqdd}

 We recall the definition of a $\GG$-equivariant DD-bundle below\cite{KrepskiWatts}.

\begin{definition} \label{d:eqDD} Let $\GG=(G_1 \rightrightarrows G_0)$ be a Lie groupoid.  
\begin{enumerate}
\item A \emph{$\GG$-equivariant DD-bundle} is a triple $(\calA, \calE, \tau)$ consisting of a DD-bundle $\calA\to G_0$, a Morita isomorphism $\calE:\partial_0^*\calA \dashto \partial_1^*\calA$ and a 2-isomorphism 
$$
\tau: \partial_2^*\calE \otimes_{\partial_0^*\partial_1^*\calA} \partial_0^*\calE \to \partial_1^* \calE
$$
 satisfying the coherence condition $\partial_2^*\tau \circ (\id\otimes \partial_0^*\tau) =\partial_1^*\tau \circ (\partial_3^*\tau \otimes \id)$. 
\item A \emph{$\GG$-equivariant Morita isomorphism} $(\calA,\calE,\tau) \dashto (\calA',\calE',\tau')$ is a pair $(\calG,\phi)$ consisting of a Morita isomorphism $\calG:\calA \dashto \calA'$ and a 2-isomorphism 
$$
\phi: \calE' \otimes_{\partial_0^* \calA'} \partial_0^*\calG \to \partial_1^*\calG \otimes_{\partial_1^*\calA} \calE
$$
 satisfying the coherence condition $(\id \otimes \tau ) \circ (\partial_2^* \phi \otimes \id) \circ (\id \otimes \partial_0^*\phi) = \partial_1^*\phi \circ (\tau' \otimes \id)$.
 \item A \emph{$\GG$-equivariant 2-isomorphism} $(\calF,\alpha) \Rightarrow (\calG,\beta)$ is a 2-isomorphism $\sigma:\calF \to \calG$ satisfying the coherence condition
 $
 \beta \circ \partial_0^*\sigma = \partial_1^*\sigma \circ \alpha.
 $
\end{enumerate}
\end{definition}

\begin{remark} \label{remark:simp}
In Definition \ref{d:eqDD} we implicitly use the simplicial identities---for example, by viewing $\partial_1^*\calE$ as a $\partial_1^*\partial_1^*\calA$-$\partial_0^*\partial_0^*\calA$ bimodule and $\partial_2^*\calE$ as a $\partial_1^*\partial_1^*\calA$-$\partial_0^*\partial_1^*\calA$ bimodule.  We continue to abuse notation in this way throughout the paper.
\end{remark}

Given a morphism $\PPhi:\mathbf{H} \to \GG$ of Lie groupoids and an $\GG$-equivariant DD-bundle $(\calA, \calE, \tau)$ we can pullback along $\PPhi$ in the obvious way to get a $\mathbf{H}$-equivariant DD-bundle $\PPhi^*(\calA,\calE,\tau)$, 
\[
\PPhi^*(\calA,\calE,\tau) = (\Phi_0^*\calA, \Phi_1^* \calE, \Phi_2^*\tau).
\]

\begin{remark}
Similar to Remark \ref{remark:simp}, to make sense of the above definition of pullback, we implicitly use (canonical) bundle isomorphisms so that, for instance, $\Phi_1^*\calE$ (a $\Phi_1^*\partial_1^*\calA-\Phi_1^*\partial_2^*\calA$ bimodule) can be viewed as a $\partial_1^*\Phi_0^* \calA-\partial_0^*\Phi_0^* \calA$ bimodule.
\end{remark}

\subsubsection{Relative Dixmier-Douady bundles} \label{ss:reldd}

Recall that given a map $f:M_1 \to M_2$, a \emph{relative DD-bundle for $f$} is a pair $(\calA,\calE)$ consisting of a DD-bundle $\calA\to M_2$ together with a Morita isomorphism $\calE: M_1 \times \CC \dashto f^*\calA$ (i.e.~a Morita trivialization of the pullback of $\calA$ along $f$).  The equivariant counterpart is defined similarly.

\begin{definition} \label{d:relDD}
Let $\PPhi:\mathbf{H} \to \GG$ be a Lie groupoid morphism.  
\begin{enumerate}
\item A \emph{relative DD-bundle for $\PPhi$} is a pair $(\calF,\alpha;\calA,\calE,\tau)$ consisting of a $\mathbf{G}$-equivariant DD-bundle $(\calA,\calE,\tau)$ together with a Morita isomorphism 
$$(\calF,\alpha):  (\CC, \CC, \id) \dashto \PPhi^*(\calA,\calE,\tau).$$  
\item A \emph{relative Morita isomorphism} $(\calF,\alpha;\calA,\calE,\tau) \dashto (\calF',\alpha';\calA',\calE',\tau')$ is a pair $(\calG,\beta;\rho)$ consisting of a $\GG$-equivariant Morita isomorphism 
$$(\calG,\beta):(\calA,\calE,\tau) \dashto (\calA',\calE',\tau')$$
 and an $\mathbf{H}$-equivariant 2-isomorphism $\rho:  \PPhi^*(\calG,\beta) \circ (\calF',\alpha')   \Rightarrow (\calF,\alpha)$.
\end{enumerate}
\end{definition}

\subsection{Differential characters} \label{ss:dc}

In \cite{KrepskiWatts}, it was shown that the bicategory of DD-bundles over a manifold is equivalent to a (strict) 2-category $\DC_1^3(M)$, the 2-category of differential characters (of degree 3) on $M$, which we briefly recall below.  (See also \cite{lerman2008differential} and \cite{hopkins2005quadratic}  for further details on the construction.)

To begin, let $\opDC_1^*$ denote the chain complex,
$$
\opDC^n_1(M) = \{(c,h,\omega) \in C^n(M;\ZZ) \times C^{n-1}(M;\ZZ) \times \Omega^n(M) \, \vert\,  \omega=0 \text{ if } n=0 \}
$$
with differential given by $d(c,h,\omega)=(dc,\omega-c-dh,d\omega)$.

An object of the 2-category $\DC_1^3(M)$, called a \emph{differential character of (degree 3) on $M$}, is a cocycle $z=(c,h,\omega)$ in $\opDC_1^3(M)$.  Given differential characters $z_1, z_2$ on $M$, a 1-arrow $z_1 \to z_2$ is a primitive $y\in \opDC_1^2(M)$ of their difference, $z_1-z_2=dy$.  Given 1-arrows, $y_1, y_2:z_1 \to z_2$, a 2-arrow $y_1 \Rightarrow y_2$ is an equivalence class $[x]\in \opDC^1_1(M)$, where $x$ is a primitive of their difference, $y_2-y_1 = dx$---see Remark \ref{r:signs} explaining the chosen sign convention.  Here we identify $x\sim x+dv$ for all $v \in \opDC^0_1(M)$. Composition is given by addition, and hence all 1- and 2- arrows are isomorphisms.

\begin{remark}
The subscript in the notation for the above chain complex and the resulting 2-category of differential characters is included for completeness---the above is the $s=1$ instance of family of chain complexes and 2-categories $\opDC_s^*$  and $\DC_s^*$, respectively.  The interested reader may consult \cite{lerman2008differential} and \cite{hopkins2005quadratic} for the significance of this parameter.
\end{remark}

\subsubsection{Equivariant differential characters} \label{ss:eqdc}

Analogous to Definition \ref{d:eqDD}, we define the 2-category $\DC_1^3(\GG)$ of $\GG$-equivariant differential characters (of degree 3) as follows.

\begin{definition} \label{d:eqdc}
Let $\GG=(G_1 \rightrightarrows G_0)$ be a Lie groupoid.
\begin{enumerate}
\item A \emph{$\GG$-equivariant differential character (of degree 3)} is a triple $(z,y,[x])$ consisting of a differential character $z$ in $\opDC_1^3(G_0)$, a 1-isomorphism $y:\partial_0^*z \to \partial_1^*z$ in $\opDC_1^2(G_1)$ and a 2-isomorphism $[x]: \partial_2^* y + \partial_0^*y \Rightarrow \partial_1^* y$, $x$ in $\opDC_1^1(G_2)$, satisfying the coherence condition that $\partial x$ is exact in $\opDC_1^1(G_3)$.
\item A \emph{1-isomorphism} $(z_1,y_1,[x_1]) \to (z_2,y_2,[x_2])$  \emph{of $\GG$-equivariant differential characters}  is a pair $(u,[v])$ where $u:z_1 \to z_2$ is in $\opDC_1^2(G_0)$ and $[v]:y_2 + \partial_0^*u \Rightarrow \partial_1^*u + y_1$ is a 2-isomorphism, with $v$ in $\opDC_1^1(G_1)$ satisfying the coherence condition that $x_1-x_2-\partial v$ is exact in $\opDC^1_1(G_2)$.
\item A \emph{$\GG$-equivariant 2-isomorphism} $(u,[v]) \Rightarrow (s,[t])$ is a 2-isomorphism $[w]:u \Rightarrow s $ satisfying the coherence condition that
 $t-v-\partial w$ is exact in $\opDC_1^1(G_1)$.
\end{enumerate}
\end{definition}

\begin{remark} \label{r:signs}
The sign conventions used here for defining 1- and 2-arrows is chosen mainly to be able to view  an object $(z,y,[x])$ in $\DC_1^3(\GG)$ as a cocycle 
\[
z\oplus y \oplus x \oplus w \in \bigoplus_{p+q=3} \opDC_1^p(G_q)
\]
(where $w$ is a primitive of $\partial x$) in the total complex of the double complex $\opDC_1(\GG)$.  
\end{remark}

\begin{lemma} \label{l:dccocycle}
$\GG$-equivariant differential characters are in one-to-one correspondence with $Z^3(\opDC_1(\GG))/\delta \opDC^0_1(G_2)$.
\end{lemma}
\begin{proof}
Note that a $\GG$-equivariant differential character amounts to a degree 3 cocycle $\xi$ in the total complex $\opDC_1^*(\GG)$,
$$
\xi=a \oplus x \oplus y \oplus z \in \opDC_1^0(G_3)\oplus\opDC_1^1(G_2)\oplus\opDC_1^2(G_1)\oplus\opDC_1^3(G_0),
$$  
where $da=\partial x$.  Changing $x$ to $x+df$ for some $f\in \opDC^0_1(G_2)$ and $a$ to $a+\partial b$ results in the \emph{same} equivariant differential character.  Therefore, $\GG$-equivariant differential characters are in one-to-one correspondence with  $Z^3(\opDC_1(\GG))/\delta \opDC^0_1(G_2)$.  
\end{proof}

\begin{lemma} \label{l:dc1arrow}
The set of 1-isomorphisms between two fixed $\GG$-equivariant differential characters is either empty or a torsor for $Z^2(\opDC_1(\GG))/\delta \opDC_1^0(G_1)$.
\end{lemma}
\begin{proof}
Similar to the previous Lemma, (if it exists) a 1-isomorphism of two fixed $\GG$-equivariant differential characters,  $(u,[v]):(z_1,y_1,[x_1]) \to (z_2,y_2,[x_2])$, amounts to a degree 2 cochain $\theta$ in the total complex $\opDC_1(\GG)$,
$$
\theta = b\oplus v \oplus u \in \opDC_1^0(G_2)\oplus\opDC_1^1(G_1)\oplus\opDC_1^2(G_0),
$$
where $db= x_2-x_1+\partial v$. Changing $v$ to $v+dg$ for some $g \in \opDC^0_1(G_1)$ and $b$ to $b+\partial g$ results in the same 1-isomorphism.  Therefore, (if non-empty) the set of 1-isomorphisms of two fixed $\GG$-equivariant differential characters is a $Z^2(\opDC_1(\GG))/\delta \opDC_1^0(G_1)$-torsor.
\end{proof}

\subsubsection{Relative differential characters}  \label{ss:relDC}
Analogous to Section \ref{ss:reldd}, we define relative differential characters for a Lie groupoid morphism as follows.

\begin{definition} \label{d:relDC}
Let $\PPhi: \mathbf{H} \to \GG$ be a Lie groupoid morphism.
\begin{enumerate}
\item A \emph{relative differential character for $\PPhi$} is a pair $(u, [v]; z, y, [x])$ consisting of a $\GG$-equivariant differential character $(z,y, [x])$ together with a trivialization $(u,[v]): (0,0,[0]) \to \PPhi^*(z,y,[x])$.
\item A \emph{relative 1-isomorphism} $(u, [v]; z, y, [x]) \to (u', [v']; z', y', [x'])$ of relative equivariant differential characters is a pair $(s,[t]; [w])$ consisting of a $\GG$-equivariant 1-isomorphism $(s,[t]):(z,y,[x]) \to (z', y', [x'])$ and an $\mathbf{H}$-equivariant 2 isomorphism $[w]: (u',[v']) + \PPhi^*(s,[t]) \Rightarrow (u,[v])$.
\end{enumerate}
\end{definition}

\subsection{The Dixmier-Douady class} \label{ss:ddclass}

Analogous to the classification of complex line bundles by their Chern class, Dixmier-Douady bundles are classified by a degree 3 cohomology class called the Dixmier-Douady class \cite{raeburn1998morita}.  
In \cite{KrepskiWatts}, it was verified that if $\GG$ is a proper Lie groupoid, then Morita isomorphism classes of $\GG$-equivariant  DD-bundles  are classified by $H^3(\GG;\ZZ)$.  In particular, an equivariant DD-bundle $(\calA,\calE,\tau)$ is Morita trivial if and only if its DD-class $DD(\calA,\calE,\tau) \in H^3(\GG;\ZZ)$ vanishes.  

\begin{remark} \label{remark:DDclassfordc}
Note that the corresponding notion of a DD-class for $\GG$-equivariant differential characters is  automatic.  Indeed, given a differential character $(z,y,[x])$, choose a cocycle $\xi$ in $\opDC_1^3(\GG)$ that represents it (see Lemma \ref{l:dccocycle}) and set $DD(z,y,[x])=[\pr (\xi)]$, where $\pr: \opDC_1^*(\GG) \to C^*(\GG)$ denotes the natural projection. This is indeed the DD-class, since isomorphism classes of differential characters are in bijection with $H^3(\opDC_1(\GG)) \cong H^3(\GG;\ZZ)$ (where the isomorphism is induced by $\pr$---see \cite{KrepskiWatts}). 
\end{remark}

\begin{remark} \label{remark:DDclassviadc}
The equivalence of bicategories \cite{KrepskiWatts}
 \[
	\mathbb{D}: \DC_1^3(\GG) \to \calBBS(\GG) 
 \]
 from $\GG$-equivariant differential characters $\DC_1^3(\GG)$ to $\GG$-equivariant DD bundles $\calBBS(\GG)$ allows for the following description of the DD-class.  Given a $\GG$-equivariant DD-bundle $(\calA,\calE,\tau)$, let $(z,y,[x])$ denote a $\GG$-equivariant differential character such that $\mathbb{D}(z,y,[x])$ is isomorphic to $(\calA,\calE,\tau)$ and set $DD(\calA,\calE,\tau)= [\pr (\xi)]$ as in Remark \ref{remark:DDclassfordc}.
 \end{remark}

We briefly describe the relative version of the DD-class. Let $\PPhi:\mathbf{H} \to \GG$ be a morphism of Lie groupoids, and suppose we are given a relative DD-bundle $(\calF,\alpha;\calA,\calE,\tau)$ for $\PPhi$. 
The relative DD-class $DD(\calF,\alpha;\calA,\calE,\tau)$ is a cohomology class in $H^3(\PPhi;\ZZ)$, the cohomology of the algebraic mapping cone of $\PPhi^*\colon C^*(\GG;\ZZ)\to C^*(\mathbf{H} ;\ZZ)$ (see \cite{weibel1995introduction} for more on relative cohomology and mapping cones).
In particular, $DD(\calF,\alpha;\calA,\calE,\tau)$ is represented by a pair
\[
( \mathbf{b}, \mathbf{c}) \in C^2(\mathbf{H};\ZZ)\oplus C^3(\GG;\ZZ),
\]
with $\delta\mathbf{c}=0$ and $[\mathbf{c}]=DD(\calA,\calE,\tau)$ as well as $\PPhi^*\mathbf{c}=-\delta \mathbf{b}$.  Similar to Remark \ref{remark:DDclassviadc}, we may define the relative DD-class using a corresponding relative $\GG$-equivariant differential character: namely, a differential character $(z,y,[x])$ in $\DC_1^3(\GG)$ and a 1-isomorphism $(u,[v]): (0,0,[0]) \to \PPhi^*(z,y,[x])$.  The relative DD-class is then obtained by setting $\mathbf{b}=\pr(\theta)$  and $\mathbf{c}= \pr(\xi)$, where $\theta$ and $\xi$ are as in the proofs of Lemmas  \ref{l:dc1arrow} and \ref{l:dccocycle}, respectively, and $\pr$ is the natural projection from Remark \ref{remark:DDclassfordc}.
\medskip

By a \emph{real} DD-class, we mean the image of the DD-class under the coefficient homomorphism induced by $\ZZ\hookrightarrow\RR$, which we will denote as $DD_\RR$.  By the de~Rham Theorem, we may also represent such real classes using differential forms, so that $DD_\RR(\calA,\calE,\tau)=[\alpha \oplus \beta\oplus\omega\oplus \eta]$ for some cocycle in the Bott-Shulman-Stasheff complex,
$$
\alpha \oplus \beta \oplus \omega\oplus \eta \in \Omega^0(G_3)\oplus \Omega^1(G_2)\oplus\Omega^2(G_1)\oplus\Omega^3(G_0).
$$
 In this case, we may say that  the $\GG$-equivariant DD-bundle $(\calA, \calE, \tau)$ \emph{represents}  $\alpha \oplus \beta\oplus\omega\oplus \eta$.

Similarly, we may use the de~Rham Theorem to write a real relative DD-class 
\[
DD_\RR(\calF,\alpha;\calA,\calE,\tau)=[(\xi \oplus \mu \oplus \nu,\alpha \oplus \beta\oplus\omega\oplus \eta)]
\]
for some cocycle $\alpha \oplus \beta\oplus\omega\oplus \eta$ as above and total form
\[
\xi \oplus \mu \oplus \nu \in \Omega^0(H_2) \oplus \Omega^1(H_1) \oplus \Omega^2(H_0)
\]
in the Bott-Shulman-Stasheff complex for $\mathbf{H}$, satisfying
\begin{equation} \label{eq:reprel}
\PPhi^*(\alpha \oplus \beta\oplus\omega\oplus \eta)=-\delta(\xi \oplus \mu \oplus \nu).
\end{equation}
  In this case, we may say that $(\calF,\alpha;\calA,\calE,\tau)$  \emph{represents the relation} \eqref{eq:reprel}.

\section{Equivariant prequantization}\label{s:prequant}

In this section we define prequantization for pre-Hamiltonian actions of quasi-presymplectic groupoids in a manner that is analogous to the definition in the quasi-Hamiltonian case \cite{meinrenken2012twisted}. 

\begin{definition} \label{d:preq}
Let $(G_1\rightrightarrows G_0,\omega\oplus \eta)$ be a quasi-presymplectic Lie groupoid, and let $(X,\omega_X,\Phi)$ be a pre-Hamiltonian $\GG$-space. A \emph{prequantization} of $(X,\omega_X,\Phi)$ consists of a relative DD-bundle for $\PPhi$  whose real $DD$-class is $[(\omega_X,\omega\oplus \eta)] \in H^3({\PPhi};\RR)$.  
\end{definition}

In other words, a prequantization of $(X,\omega_X,\Phi)$ is a relative DD-bundle $(\calF,\alpha; \calA, \calE, \tau)$ that represents the pre-Hamiltonian condition $\PPhi^*(\omega\oplus\eta) = -\delta \omega_X$.

\begin{remark}\label{remark:noconnex}
In contrast to other definitions of prequantization appearing in the literature, Definition \ref{d:preq} does not involve a choice of connective structure.  
\end{remark}

As with other treatments of prequantization, an obstruction to the existence of a prequantization can be characterized as an \emph{integrality} condition.  

\begin{definition} 
For a morphism $\PPhi:\mathbf{H} \to \GG$ of Lie groupoids, a relative closed total form $(\alpha,\beta)$ in $\Omega^3(\PPhi)$ is \emph{integral} whenever $[(\alpha,\beta)] \in H^3(\PPhi;\RR)$ is in the image of the coefficient homomorphism $H^3(\PPhi;\ZZ) \to H^3(\PPhi;\RR)$.
\end{definition}

\begin{proposition} \label{p:integrality}
Let $(G_1\rightrightarrows G_0,\omega\oplus \eta)$ be a quasi-presymplectic Lie groupoid, and let $(X,\omega_X,\Phi)$ be a pre-Hamiltonian $\GG$-space. A prequantization of $(X,\omega_X,\Phi)$ exists if and only if $(\omega_X,\omega\oplus \eta)$ is an integral (relative) total form.  
\end{proposition}
\begin{proof}
It is clear that if a prequantization exists that $(\omega_X,\omega\oplus \eta)$ is an integral relative total form.

Conversely, suppose  $(\omega_X,\omega\oplus \eta)$ is integral.  Then there exists $(\mathbf{b}, \mathbf{c}) \in C^2(\GG\ltimes X;\ZZ)\oplus C^3(\GG;\ZZ)$ satisfying $\delta\mathbf{c}=0$ and  $\PPhi^*\mathbf{c}=-\delta \mathbf{b}$, as well as
\begin{align}
\omega_X -\mathbf{b} &= \delta \mathbf{g} + \PPhi^*\mathbf{h} \label{eq:reldeg2} \\ 
\omega\oplus \eta - \mathbf{c} &= -\delta \mathbf{h} \label{eq:reldeg3}
\end{align}
for some $(\mathbf{g},\mathbf{h}) \in C^2(\PPhi;\RR)$. 
Equation \eqref{eq:reldeg3} says that $\omega\oplus \eta$ is integral and it is straightforward to verify that 
$(z,y,[x])=(c_0,-h_0,\eta; c_1,h_1,\omega; [c_2,-h_2,0])$  defines a  $\GG$-equivariant differential character and hence $\GG$-equivariant DD-bundle $(\calA,\calE,\tau)$ with DD-class equal to $[\mathbf{c}]$, and real DD-class $[\omega\oplus \eta]$.

Equation \eqref{eq:reldeg2}  and the condition that  $(\mathbf{b},\mathbf{c})$ is a relative cocycle gives the following system:
\begin{align*}
\Phi^* c_0&=-db_0 & \Phi^* h_0 &= \omega_X - b_0 -dg_0 \\
\Phi^* c_1&= -\partial b_0 + db_1 & \Phi^* h_1 &= -b_1 -\partial g_0 + dg_1 \\
\Phi^*c_2 &= -\partial b_1 - db_2 & \Phi^* h_2 &=-b_2 -\partial g_1 \\
\Phi^*c_3 &=- \partial b_2 & & 
\end{align*}
Together with the condition $\delta\omega_X + \PPhi^*(\omega \oplus \eta)=0$, we find that 
$$
(u,[v])=(b_0, g_0, \omega_X;[-b_1,g_1,0])
$$
 is a trivialization $(0,0,[0]) \to \PPhi^*(z,y,[x])$, which gives a Morita trivialization 
$$
(\calF,\alpha):  (\CC,\CC,\id) \dashto  \PPhi^*(\calA,\calE,\tau).
$$
  It follows that the relative DD-bundle $(\calF,\alpha;\calA,\calE,\tau)$ is the desired prequantization.
\end{proof}

\subsection{Relation with Laurent-Gengoux \& Xu's definition}\label{ss:compare}

The definition of prequantization in Definition \ref{d:preq} is readily seen to agree with the definition of prequantization for quasi-Hamiltonian $G$-space with $G$-valued moment map \cite{meinrenken2012twisted}.  In this subsection, we show that Definition \ref{d:preq} is consistent with the definition in \cite{laurent2005quantization}, which employs $S^1$-central extensions to model $S^1$-gerbes.  For convenience, we recall some elementary definitions regarding $S^1$-central extensions and  prequantization from \cite{laurent2005quantization}. (For more further on $S^1$-central extensions, the reader may consult \cite{behrend2011differentiable} or \cite{tu2004twisted}.)

\begin{definition} \label{d:s1centext}
   An \emph{$S^1$-central extension} of a Lie groupoid $G_1 \rightrightarrows G_0$ is a Lie groupoid $R\rightrightarrows G_0$ together with a Lie groupoid homomorphism
\[
\xymatrix{
R \ar@<-.7ex>[d] \ar@<+.7ex>[d] \ar[r]^{\pi} & G_1 \ar@<-.7ex>[d] \ar@<+.7ex>[d]\\
G_0 \ar@{=}[r] & G_0
}
\]
with the property that $\pi$ is a (left) principal $S^1$-bundle and the principal $S^1$-action on $R$ is compatible with the groupoid multiplication on $R$: $(s\cdot x)(t\cdot y) = st \cdot (xy)$ for all $s,t \in S^1$ and $(x,y) \in R_2=R\times_{G_0} R$.  Such an $S^1$-central extension is denoted by $R\to G_1 \toto G_0$.
\end{definition}

Recall from \cite{behrend2011differentiable} that associated to an $S^1$-central extension $R\to G_1 \toto G_0$ is a DD-class $DD(R)$ in $H^3(\GG;\ZZ)$, which by  \cite[Proposition 4.7]{behrend2011differentiable} pulls back along $\epsilon_{\GG}$ to zero in $H^3(G_0;\ZZ)$.  In particular, if $(\GG,\omega\oplus \eta)$ is a quasi-presymplectic groupoid and $R\to G_1 \toto G_0$ is an $S^1$-central extension whose real DD-class is $[\omega\oplus \eta]$, then $\eta$ must be exact.

For a quasi-presymplectic groupoid $(\GG,\omega\oplus \eta)$ with exact 3-form $\eta$, the authors in \cite{laurent2005quantization} define a prequantization of a Hamiltonian $\GG$-space $(X,\omega_X,\Phi)$ as an $S^1$-central extension $R\to G_1 \toto G_0$ of $\GG$ together with  a principal $S^1$-bundle $p:L\to X$ equipped with an $R$-action satisfying $(s\cdot r) (t \cdot y) = st \cdot (ry)$ for all $s,t \in S^1$ and $(r,y)\in R\times_{G_0} L$.  Theorem \ref{t:compare}(1) below shows that this definition is compatible with Definition \ref{d:preq}.  

The proof of Theorem \ref{t:compare} requires the following lemma, which describes the coherence condition for 1-isomorphisms of equivariant DD-bundles (see Definition \ref{d:eqDD}(2)) in the special case of a trivialization.  Note that the statement of the lemma implicitly abuses notation in the spirit of Remark \ref{r:reorder}.

\begin{lemma} \label{l:coherence_triv}
Let $\GG$ be a Lie groupoid, and let $(\calA,\calE,\tau)$ be a $\GG$-equivariant DD-bundle.  A trivialization $(\calG,\beta):\epsilon^*(\calA,\calE,\tau) \dashto (\CC,\CC,\id)$ gives rise to an equality of 2-isomorphisms $\beta = \epsilon^*\tau$.  
\end{lemma}
\begin{proof}
The coherence condition from Definition \ref{d:eqDD}(2) amounts to the following commutative diagram of bundle isomorphisms over $G_0$:

\[
\xymatrixcolsep{2pc}
\xymatrix{
	    & \epsilon^*\calE \otimes \epsilon^*\calE \otimes \calG \ar[rr]^-{\id\otimes \beta} \ar[dl]_-{\epsilon^*\tau \otimes\id} &		& \epsilon^*\calE\otimes \calG \otimes \CC \ar[dr]^-{\beta \otimes \id}	&	\\
\epsilon^*\calE \otimes \calG \ar[drr]_-{\beta} &			&		&		& \calG \otimes \CC\otimes \CC \ar[dll] \\
	    &			& \calG \otimes \CC		&		&	\\
}
\]
where the unlabelled arrow is a canonical isomorphism.  Hence, the diagram
\[
\xymatrix{
\epsilon^*\calE \otimes \epsilon^*\calE \otimes \calG \ar[r]^-{\id\otimes \beta} \ar[d]_-{\epsilon^*\tau \otimes\id} 	& \epsilon^*\calE\otimes \calG \otimes \CC  \\
\epsilon^*\calE \otimes \calG \ar[ur]
}
\]
commutes, where the unlabelled arrow is a canonical isomorphism.
Therefore,  the 2-isomorphism $\beta$ (see Remark \ref{r:reorder}), which may be written as the composition
\[
\xymatrix{
\epsilon^*\calE  \ar[r]  & \epsilon^*\calE^* \otimes \epsilon^*\calE \otimes \epsilon^*\calE \otimes \calG \otimes \calG^* \ar[rr]^-{\id\otimes \id \otimes \beta \otimes \id}  & & \epsilon^*\calE^* \otimes \epsilon^*\calE\otimes \calG \otimes \CC  \otimes \calG^* 
\ar@(d,u)[]!<10ex,-2ex>;[dll]!<1ex,2ex>\\
& \epsilon^*\calE^* \otimes \epsilon^*\calE \otimes \calG \otimes \calG^* \ar[r]  & \calA &
}
\]
(where unlabelled arrows are canonical isomorphisms)
coincides with $\epsilon^*\tau$, which may be written as the composition
\[
\xymatrix{
\epsilon^*\calE  \ar[r]  & \epsilon^*\calE^* \otimes \epsilon^*\calE \otimes \epsilon^*\calE \otimes \calG \otimes \calG^* \ar[rr]^-{\id\otimes \epsilon^*\tau\otimes \id \otimes \id}  
& & \epsilon^*\calE^* \otimes \epsilon^*\calE \otimes \calG \otimes \calG^* \ar[r]  & \calA
}
\]
(where unlabeled arrows are canonical isomorphisms),
as required.
\end{proof}

\begin{theorem} \label{t:compare}
Let $(\GG,\omega\oplus \eta)$ be a quasi-presymplectic groupoid, and let $(X,\omega_X,\Phi)$ be a pre-Hamiltonian $\GG$-space.  Suppose that $(\calF,\alpha ; \calA, \calE, \tau)$ is a prequantization of $(X,\omega_X,\Phi)$ and that $\epsilon_{\GG}^*(\calA,\calE,\tau)$ admits a trivialization over $G_0$ (viewed as a trivial groupoid).  Then 
\begin{enumerate}
	\item there exists an $S^1$-central extension $R\stackrel{\pi}{\longrightarrow} G_1 \toto G_0$ and a principal $S^1$-bundle $p:L\to X$ equipped with an $\mathbf{R}$-action with anchor $\Phi\circ p$, satisfying $(s\cdot r) (t \cdot y) = st \cdot (ry)$ for all $s,t \in S^1$ and $(r,y)\in R\times_{G_0} L$;
	\item the $\mathbf{R}$-equivariant curvature class of $L$ is $[(\Phi^* \beta_0 -\omega_X) \oplus \Phi^*\theta] \in H^2(\mathbf{R}\ltimes X;\RR)$, where $\beta_0$ is a primitive of $\eta$ and $\theta$ is a primitive of the curvature form of $\pi:R \to G_1$ pulled back along $\pi$.
\end{enumerate}
\end{theorem}
\begin{proof}[Proof of (1)]
Let $(\calG,\beta):(\CC,\CC,\id) \dashto \epsilon_{\GG}^*(\calA,\calE,\tau)$ be a trivialization over $G_0$.  Consider the composition $\CC \dashto \partial_0^*\calA \dashto \partial_1^*\calA \dashto \CC$ given by the line bundle
$
\mathcal{R} = \partial_1^*\calG^{*} \otimes_{\partial_1^*\calA} \calE \otimes_{\partial_0^*\calA} \partial_0^*\calG
$
over $G_1$.
Let $R\to G_1$ be its associated principal $S^1$-bundle.  Then $R\to \G_1 \toto G_0$ is an $S^1$-central extension of $\GG$.

Let $(\calF,\alpha)$ be the trivialization of $\Phi^*(\calA,\calE,\tau)$ given by the prequantization.
Consider the composition $\CC \dashto \Phi^*\calA \dashto \CC$ given by the line bundle
$
\mathcal{L}=\calF^* \otimes_{\Phi^*\calA} \Phi^*\calG
$
over $X$.  Let $p:L\to X$ be its associated principal $S^1$-bundle.

The 2-isomorphism $\alpha$  (see Remark \ref{r:reorder}) in the diagram 
\[
\xymatrix{
\CC \ar@{-->}[rr]^-{\Phi^*\partial_0^* \calG} && \partial_0^* \Phi^* \calA \ar@{-->}[rr]^{\Phi^*\calE} \ar@{-->}[dr]_{\partial_0^*\calF^*} \rrtwocell<\omit>{<3>\alpha}&  & \partial_1^* \Phi^* \calA \ar@{-->}[rr]^-{\Phi^*\partial_1^* \calG^*} && \CC \\
&& & \CC \ar@{-->}[ur]_{\partial_1^*\calF} & & &
}
\]
gives rise to a bundle isomorphism  $\Phi^*\mathcal{R} \cong \partial \mathcal{L}=\partial_0^*\mathcal{L} \otimes \partial_1^* \mathcal{L}^*$, or equivalently $\Phi^*R \cong \partial L$.  We verify below that this isomorphism, which will be denoted by $\psi:\Phi^*R=R\times_{G_0} X \to \partial L$, defines an action 
$$
R\times_{G_0} L \to L, \quad (r,y) \mapsto r\cdot y,
$$
with anchor map $\Phi\circ p:L \to G_0$.

 It is straightforward to verify that 
 \begin{align*}
\partial L 
&= \{(\gamma, x, y_0, y_1) \in G_1\times X \times L \times L\, | \, x=p(y_0), \gamma\cdot x = p(y_1)\} \big/ \sim
 \end{align*}
 where $(\gamma,x,y_0,y_1) \sim (\gamma,x,\lambda y_0, \lambda y_1)$ for $\lambda \in S^1$. 
 (We shall use brackets $[\,\cdots]$ to denote $\sim$-equivalence classes.)
 Given $(r,y)\in R\times_{G_0} L$, define $r\cdot y \in L$ by the equality
 \[
 \psi (r,p(y)) = [\pi(r),p(y),y,r\cdot y].
 \]
This is well-defined since the action of $S^1$ on $L$ is free.  
We check that this formula defines an action of $\mathbf{R}$ on $L$ with anchor $\Phi  p$.  That $\Phi p(r\cdot y)=\partial_1(r)$ follows immediately from the fact that $\GG$ acts on $X$: since $p(r\cdot y) = \pi(r) \cdot p(y)$, and hence $\Phi p(r\cdot y) = \Phi(\pi(r) \cdot p(y)) = \partial_1(\pi(r)) = \partial_1(r)$.  

Recall that the unit map $\epsilon_{\mathbf{R}}:G_0 \to R$  can be viewed as a section of the $S^1$-bundle $\epsilon_{\GG}^*R\to G_0$.
We show next that the units of $\mathbf{R}$ act trivially.  This will follow from the commutative diagram \eqref{d:compat} below of $S^1$-bundles over $X$.  (The commutativity of the diagram will be shown subsequently.)
\begin{equation} \label{d:compat}
\xymatrix{
\epsilon_{\GG\ltimes X}^*\Phi^*R \ar[r]^{\epsilon_{\GG\ltimes X}^*\psi} \ar@{=}[d] & \epsilon_{\GG\ltimes X}^*\partial L \ar@{=}[d] \\
\Phi^*\epsilon_{\GG}^*R \ar[r] & X\times S^1
}
\end{equation}
The vertical maps marked as equalities are canonical isomorphisms.  The canonical isomorphism on the right is 
\[
[\epsilon_{\GG}(\Phi(x)),x,y_0,y_1] \mapsto (x,\lambda (y_0,y_1))
\]
 where $\lambda(y_0,y_1) \in S^1$ is defined by $y_1=\lambda(y_1,y_1)y_0$.  The trivialization $\Phi^*\epsilon_{\GG}^*R \to X\times S^1$ sends $(r,x) \mapsto \lambda(w(\Phi(x)),r)$. 
 
 To see that \eqref{d:compat} guarantees that the units act trivially, let $a\in G_0$, and consider $\epsilon_{\mathbf{R}}(a) \in R$ and $y\in L$ satisfying $\Phi(p(y))=\partial_0 (\epsilon_{\mathbf{R}}(a))=a$.  Then $(\epsilon_{\mathbf{R}}(a),p(y)) \in \epsilon_{\GG\ltimes X}^*\Phi^*R = \Phi^*\epsilon_{\GG}^*R$ and the commutativity of the diagram forces
 \[
 \lambda(y,\epsilon_{\mathbf{R}}(a)\cdot y)= \lambda(w(\Phi(p(y)),\epsilon_{\mathbf{R}}(a)))=\lambda(\epsilon_{\mathbf{R}}(a),\epsilon_{\mathbf{R}}(a))=1,
 \]
 as required.

We now verify the commutativity of \eqref{d:compat}, which ultimately relies on  the coherence conditions (see Definition \ref{d:eqDD}) satisfied by the trivializations $(\calF,\alpha)$ and $(\calG,\beta)$.   Omitting certain canonical isomorphisms, diagram \eqref{d:compat} may be rewritten as
\begin{equation} \label{d:compat2}
\xymatrixcolsep{5pc}
\xymatrix{
\Phi^* \calG^* \otimes_{\Phi^*\calA} \epsilon_{\GG\ltimes X}^*\Phi^* \calE \otimes_{\Phi^*\calA} \Phi^*\calG \ar[r]^-{\id\otimes \epsilon_{\GG\ltimes X}^*\alpha \otimes \id} \ar@{=}[d] & 
\Phi^* \calG^* \otimes_{\Phi^*\calA} \calF \otimes_{\CC} \calF^* \otimes_{\Phi^*\calA} \Phi^*\calG \ar@{=}[d] \\
\Phi^* \calG^* \otimes_{\Phi^*\calA} \Phi^* \epsilon_{\GG}^*\calE \otimes_{\Phi^*\calA} \Phi^*\calG \ar[r] & \CC
}
\end{equation}
where the vertical maps marked as equalities are canonical isomorphisms. The lower horizontal map is the composition,
\[
\xymatrixcolsep{5pc}
\xymatrix{
\Phi^* \calG^* \otimes_{\Phi^*\calA} \Phi^* \epsilon_{\GG}^*\calE \otimes_{\Phi^*\calA} \Phi^*\calG \ar[r]^-{\id\otimes \Phi^*\beta \otimes \id} & \Phi^* \calG^* \otimes_{\Phi^*\calA}  {\Phi^*\calA}  \otimes_{\Phi^*\calA} \Phi^*\calG \cong \CC.
}
\]

Recall from Definition \ref{d:eqDD} that the 2-isomorphisms $\epsilon_{\GG\ltimes X}^*\alpha$ and $\Phi^*\beta$ must satisfy certain coherence conditions.  Here---see Remark \ref{r:reorder}---they are being viewed as 2-isomorphisms
\[
\epsilon_{\GG\ltimes X}^*\alpha: \epsilon_{\GG\ltimes X}^*\Phi^*\calE \to  \calF \otimes_{\CC} \calF^*, \quad \text{and}\quad \Phi^*\beta: \Phi^*\epsilon_{\GG}^*\calE \to \Phi^*\calG \otimes_\CC \Phi^*\calG^* \cong \Phi^*\calA. 
\]
Writing the 2-isomorphisms $\Phi^*\epsilon_{\GG}^*\tau$ and $\epsilon_{\GG\ltimes X}^*\Phi^*\tau$ as 
\[
\Phi^*\epsilon_{\GG}^*\tau: \Phi^*\epsilon_{\GG}^*\calE \to \Phi^*\calA, \quad \text{and}\quad \epsilon_{\GG\ltimes X}^*\Phi^*\tau: \epsilon_{\GG\ltimes X}^*\Phi^* \calE \to \Phi^* \calA,
\]
the corresponding coherence conditions are
\[
\Phi^*\epsilon_{\GG}^*\tau = \Phi^*\beta, \quad \text{and}\quad \epsilon_{\GG\ltimes X}^*\Phi^*\tau =\epsilon_{\GG\ltimes X}^* \alpha,
\]
by Lemma \ref{l:coherence_triv}.  Therefore, the trivializations   in diagram \eqref{d:compat2} coincide (i.e.~the diagram commutes) because each is a trivialization that is compatible with $\Phi^*\epsilon_{\GG}^*\tau=\epsilon_{\GG\ltimes X}^*\Phi^*\tau$.

Next, we show that  $r_0\cdot (r_1\cdot y) = (r_0r_1)\cdot y$ for compatible $r_0,r_1 \in R$, $y\in L$.  As will be verified next, this follows from the commutativity of the diagram below \[
 \xymatrix{
 \partial \Phi^*R \ar[r]^{\partial \psi} \ar@{=}[d] & \partial^2L \ar@{=}[d] \\
 \Phi^*\partial R \ar[r] & (\GG\ltimes X)_2\times S^1
 }
 \]
  (where the equalities are canonical isomorphisms).  (The commutativity of the diagram follows by an argument similar to the one given above for diagram \eqref{d:compat}, ultimately relying on the coherence condition satisfied by the trivialization $(\calF,\alpha)$.)

A straightforward computation gives the identifications
  \[
  \partial \Phi^* R = \Phi^*\partial R = \{ 
  (r_0,r_1,r,x) \, \vert \, \pi(r_0)\pi(r_1)=\pi(r), \partial_0 (r_1)=\Phi(x)
  \} \big/ \sim
  \]
  where $(r_0,r_1,r,x) \sim (\lambda r_0,  r_1, \lambda r, x) \sim (r_0, \mu r_1, \mu r, x)$, for $\lambda,\mu \in S^1$, and 
  \[
  \partial^2L = \left\{ 
  (\gamma_0, \gamma_1, x; y_0, y_1, y_0', y_1', y_0'', y_1'') \, \left|
\begin{array}{c}
  \partial_0 (\gamma_1) = \Phi(x) ,\, \partial_0 (\gamma_0) = \Phi(\gamma_1\cdot x) \\
  x=p(y_0)=p(y_0') \\
  \gamma_1 \cdot x = p(y_1)=p(y_0'') \\
  \gamma_0\gamma_1 \cdot x = p(y_1')=p(y_1'') 
\end{array}
\right.
  \right\} \Big/ \sim
  \]
  where
\[
  (\gamma_0, \gamma_1, x; y_0, y_1, y_0', y_1', y_0'', y_1'') \sim 
  (\gamma_0, \gamma_1, x; \lambda y_0, \lambda \alpha y_1, \mu y_0',\mu \alpha \beta  y_1', \nu y_0'', \nu \beta y_1'') 
\]
  for $\lambda,\mu,\nu,\alpha,\beta \in S^1$.
For $[r_0,r_1,r_0r_1,p(y)] \in \partial \Phi^*R$, we compare the two ways of going around the above commutative diagram:
\[
\xymatrix{
[r_0,r_1,r_0r_1,p(y)] \ar@{|->}[r] \ar@{|->}[dd] & [\pi(r_0), \pi(r_1),p(y); y, r_1 \cdot y, y, (r_0r_1)\cdot y, r_1\cdot y, r_0\cdot(r_1\cdot y)] \ar@{|->}[d] \\
 & (\pi(r_0), \pi(r_1), p(y); \lambda((r_0r_1)\cdot y, r_0\cdot (r_1 \cdot y))) \ar@{=}[d] \\
 [r_0,r_1,r_0r_1,p(y)] \ar@{|->}[r] & (\pi(r_0), \pi(r_1), p(y); \underbrace{\lambda(r_0r_1, r_0 r_1)}_{=1}) 
}
\]  
(where $\lambda(-,-)$ is defined analogously as above) and hence $r_0\cdot (r_1\cdot y) = (r_0r_1)\cdot y$, as required.

Finally, that the $\mathbf{R}$-action is compatible with the $S^1$-action, as stated in the Theorem, follows from the fact that $\psi$ is a bundle map.
\end{proof}

\begin{proof}[Proof of (2)]
In addition to making use of the equivalence (of bicategories) between equivariant DD-bundles and equivariant differential characters of degree 3 \cite{KrepskiWatts}, this proof also makes use of the degree 2 version shown in \cite{lerman2008differential} for principal $S^1$-bundles (the reader is referred there for details).  Briefly, we recall that for a Lie groupoid $H_1 \toto H_0$,   the category of $\mathbf{H}$-equivariant $S^1$-bundles is equivalent to the category $\DC_1^2(\mathbf{H})$ of $\mathbf{H}$-equivariant differential characters of degree 2.  An object in $\DC_1^2(\mathbf{H})$ is  a pair $(z,[y])$ consisting of a differential character (i.e. cocycle) $z \in \opDC_1^2(H_0)$ and a 1-isomorphism $[y]:\partial_0^*z \to \partial_1^*z$ (i.e.~$y \in  \opDC^1_1(H_1)$ is a primitive of $\partial_1^*z-\partial_0^*z$). (In this case, two 1-isomorphisms whose difference is exact are considered equal.)  Moreover, $y$ must satisfy a coherence condition that $\partial y$ is exact.

Let $\zeta=(\zeta_0, \zeta_1, [\zeta_2])=((c_0,h_0,\eta),(c_1,h_1,\omega),[c_2,h_2,0])$ be a differential character (of degree 3) corresponding to $(\calA,\calE,\tau)$, and let $((b_0,g_0,\beta_0),[b_1,g_1,0]):0\to \epsilon_{\GG}^*\zeta$ be a trivialization.  Note that one of the defining relations of this trivialization is that $\eta = d\beta_0$.

Let  $((B_0,F_0,\omega_X),[B_1,F_1,0]):0\to \Phi^*\zeta $ be a prequantization.  
Then it is straightforward to verify that the principal $S^1$-bundle $\pi:R\to G_1$ may be represented by the differential character of degree 2 (also denoted by $R$),
\[
R=(\partial b_0 +c_1, \partial g_0+h_1, \partial \beta_0+\omega),
\]
and that the $S^1$-bundle $p:L\to X$ may be represented by the differential character of degree 2 in $\DC_1^2(X)$ (also denoted by $L$),
\[
L=(\Phi^*b_0-B_0,\Phi^*g_0-F_0,\Phi^*\beta_0-\omega_X).
\]
This immediately gives the non-equivariant curvature class of $L$ as $[\Phi^*\beta_0-\omega_X]$.  

Next we will find an equivariant extension for $L$ in $\DC_1^2(\mathbf{R}\ltimes X)$.  That is, we will find a 1-isomorphism $[y]:\partial_0^*L \to \partial_1^*L$ (so that $\partial L = dy$) satisfying the coherence condition that $\partial y$ is exact.  This will exhibit the desired object $(L,[y])$ in $\DC_1^2(\mathbf{R}\ltimes X)$.

To introduce some notation, consider the commutative diagram of $S^1$-central extensions below.
\[
\xymatrix{
R \times_{G_0} X \ar[r]^-{\widehat{\Phi}} \ar[d]_{\pi'} & R \ar[d]^\pi \\
G_1 \times_{G_0} X \ar[r]^-{\Phi_1}  \ar@<-.7ex>[d] \ar@<+.7ex>[d] & G_1  \ar@<-.7ex>[d] \ar@<+.7ex>[d] \\
X \ar[r]^{\Phi} & G_0
}
\]

Since
\begin{align*}
\partial_{\mathbf{R} \ltimes X} L &= (\pi')^*\partial_{\GG\ltimes X} L  \\
&= (\pi')^* ({\Phi}_1^*R + d(B_1,F_1, 0)) \\
&= \widehat{\Phi}^*(\pi^*R) + (\pi')^*(d(B_1, F_1, 0)), 
\end{align*}
it would suffice to find a primitive $r$ with $\pi^*R=dr$ such that $\partial_{\mathbf{R}\ltimes X}(\widehat{\Phi}^*r  + (\pi')^*(B_1,F_1,0))$ is exact.  (In this case, we set $y=\widehat{\Phi}^*r+(\pi')^*(B_1,F_1,0)$ to get the data needed for the desired  equivariant extension $(L,[y])$.)  We verify below that such a primitive exists.

Since the pullback $S^1$-central extension $\pi^*R \to R \toto G_0$ is canonically trivial, it follows that $\pi^*\zeta$ is canonically trivial as well.  That is, there exists a \emph{canonical} 1-isomorphism $(u,[v]):0 \to \pi^*\zeta$, where 
\begin{equation} \label{e:can1iso}
-\pi^*\zeta_0=du, \quad -\pi^*\zeta_1-\partial_{\mathbf{R}} u = dv, \quad \text{and} \, \, -\pi^*\zeta_2-\partial_{\mathbf{R}} v \,\, \text{ is exact.}
\end{equation}
To be more precise, the trivialization $\sigma=\sigma_{\mathbf{R}}=(u,[v])$ above satisfies the following property.
Given a groupoid morphism $f:\mathbf{H}\to \GG$, let $f^*R \stackrel{p}{\longrightarrow} H_1 \toto H_0$ be the pullback $S^1$-central extension,
\[
\xymatrix{
f^*R \ar[r]^{\hat{f}} \ar[d]_p & R \ar[d]^\pi \\
H_1 \ar[r]^f \ar@<-.7ex>[d] \ar@<+.7ex>[d] & G_1 \ar@<-.7ex>[d] \ar@<+.7ex>[d]  \\
H_0 \ar[r]^{f_0} & G_0 \\
}
\]
 and consider the pullback central extension $\hat{f}^*(\pi^*R)\cong p^*(f^*R) \to f^*R \toto H_0$,   
\[
\xymatrix{
\hat{f}^*(\pi^*R) \ar[r]  \ar[d] & \pi^*R \ar[d]^\pi \\
f^*R \ar[r]^{\hat{f}} \ar@<-.7ex>[d] \ar@<+.7ex>[d] & R \ar@<-.7ex>[d] \ar@<+.7ex>[d]  \\
H_0 \ar[r]^{f_0} & G_0. \\
}
\]
Viewing trivializations of $S^1$-central extensions as sections, let $\sigma:{R} \to  \pi^*R$ be a section of $\pi$.  Then the  sections 
$$
\sigma_{f^*R},\, \hat{f}^*\sigma_R:f^*R \to \hat{f}^*(\pi^*R)
$$ 
agree (up to natural 2-isomorphism). 

Taking $f=\epsilon_\GG$, it follows that $\epsilon_\GG^*(u,[v])=(u,[\epsilon_\GG^*v])$ and $((b_0, g_0, \beta_0),[b_1, g_1, 0])$ are 2-isomorphic.  In particular, 
\(
u=(b_0, g_0, \beta_0)-dw, 
\)
for some $w$ in $\opDC_1^1(G_0)$.  Together with \eqref{e:can1iso}, a straightforward calculation verifies that   $r=\partial_R w-v$ will suffice. Letting $\theta$ be the component of $r\in \opDC^1_1(R)$ in the $\Omega^1(R)$ summand, we find that the equivariant curvature class is as stated in the Theorem.
\end{proof}

As a special case of Theorem \ref{t:compare} above, we briefly consider the case of ordinary Hamiltonian $G$-spaces.  (See also \cite[Example 4.3]{laurent2005quantization}.)

\begin{example} \label{eg:ordinarypreq} Let $(M,\omega)$ be a symplectic manifold equipped with a Hamiltonian $G$-action with $G$-equivariant moment map $\Phi:M\to \g^*$, where $G$ is a compact Lie group.  In other words, $(M,\omega,\Phi)$ is a Hamiltonian $T^*G$-space, for the symplectic groupoid $(T^*G\toto \g^*,-d\alpha)$, where $\alpha$ is the canonical 1-form on the cotangent bundle (which satisfies $\partial \alpha=0$).  Additionally, suppose $(M,\omega,\Phi)$ admits a prequantization.  Since $[-d\alpha]=[-\delta \alpha]=0$ in $H^3(T^*G \toto \g^*;\RR)$, we may choose the trivial $S^1$-central extension $\mathbf{R}=T^*G \times S^1 \to T^*G \toto \g^*$ in Theorem \ref{t:compare}.  (In the notation of the statement of the theorem, we have $\beta_0=0$ and $\theta=\pi^*\alpha$.)
Moreover, the $\mathbf{R}$-equivariant $S^1$-bundle $p:L\to M$ has (real) curvature class $[-\omega ] \in H^2(M,;\RR)$---i.e.~$L$ is a prequantum circle bundle---and its corresponding $\mathbf{R}$-equivariant extension is $[-\omega\oplus \Phi^*\pi^*\alpha] \in H^2(\mathbf{R}\ltimes M;\RR)$.  

Since $\mathbf{R}$ is a trivial $S^1$-central extension, we may view $L$ as a $G$-equivariant $S^1$-bundle.  Hence, the $G$-equivariant curvature is $[-\omega\oplus \Phi^*\alpha]$ in $H^2(G\ltimes M;\RR)$, which
corresponds to $[-\omega+\Phi] \in H^2_G(M;\RR)$ in the Cartan model for $G$-equivariant de~Rham cohomology.
\eoe
 \end{example}

\section{Morita invariance}\label{s:morita}

This section considers the compatibility of  Definition \ref{d:preq} with Morita equivalence, and establishes a Morita invariance property for prequantization in Theorem \ref{t:moritapreq}.

\subsection{Morita equivalence of quasi-presymplectic groupoids}

 Recall that a morphism of Lie groupoids $\mathbf{F}:\GG \to \mathbf{H}$ is a \emph{(weak) equivalence} if the map 
 \[
 t\circ \pr_1:H_1 \times_{H_0} G_0 \to H_0
 \]
 (defined on pairs $(h,x)$ satisfying $s(h)=F_0(x)$)  is a surjective submersion and if the commutative square below is cartesian:
\[
\xymatrix{
G_1 \ar[r]^-{(s,t)} \ar[d]_{F_1} & G_0 \times G_0 \ar[d]^{F_0 \times F_0} \\
H_1 \ar[r]^-{(s,t)} & H_0 \times H_0
}
\]
Also, recall that Lie groupoids $\GG$ and $\mathbf{H}$ are \emph{Morita equivalent} if there exists a Lie groupoid $\mathbf{K}$ and a pair of equivalences $\xymatrix{\GG & \ar[l]_{\lambda} \mathbf{K} \ar[r]^{\rho} & \mathbf{H}}$.  In this case, we shall refer to the pair $(\lambda,\rho)$ as a \emph{Morita equivalence}.

\begin{definition} \label{d:moritaeq}
Let $(\GG,\omega\oplus \eta)$ and $(\mathbf{H},\nu\oplus\chi)$ be quasi-presymplectic groupoids. A Morita equivalence $\xymatrix{\GG & \ar[l]_{\lambda} \mathbf{K} \ar[r]^{\rho} & \mathbf{H}}$ is an \emph{equivalence of quasi-presymplectic groupoids} if the equality $[\lambda^*(\omega\oplus \eta)] = [\rho^*(\nu \oplus \chi)]$ holds in $H^3(\mathbf{K};\RR)$.
\end{definition}

\begin{remark} \label{r:firstkind}
If $\xymatrix{\GG & \ar[l]_{\lambda} \mathbf{K} \ar[r]^{\rho} & \mathbf{H}}$ is an equivalence  of proper quasi-presymplectic groupoids $(\GG,\omega\oplus \eta)$ and $(\mathbf{H},\nu\oplus\chi)$, then one may choose a primitive $\alpha \in \Omega^2(K_0)$ with 
\begin{equation} \label{eq:3forms}
\rho^*(\nu\oplus \chi) - \lambda^*(\omega \oplus \eta) = \delta \alpha.
\end{equation}
That is, $\mathbf{K}$ may be viewed as a quasi-presymplectic groupoid in two ways, $(\mathbf{K},\lambda^*(\omega\oplus \eta))$ and $(\mathbf{K},\rho^*(\nu \oplus \chi))$, which differ by a gauge transformation \cite{xu2004momentum}.
\end{remark}

Recall that two Morita equivalences $\xymatrix{\GG & \ar[l]_{\lambda} \mathbf{K} \ar[r]^{\rho} & \mathbf{H}}$ and $\xymatrix{\GG & \ar[l]_{\lambda'} \mathbf{K'} \ar[r]^{\rho'} & \mathbf{H}}$ are \emph{2-isomorphic} if there exists a Morita equivalence $\xymatrix{\mathbf{K} & \ar[l] \mathbf{L} \ar[r] & \mathbf{K}'}$ such that the following diagram commutes:
\[
\xymatrix{
	&	\mathbf{K} \ar[dl]_{\lambda} \ar[dr]^{\rho}	&	\\
\GG	&	\mathbf{L} \ar[u] \ar[d] 	& \mathbf{H} \\
	&	\mathbf{K'} \ar[ul]^{\lambda'} \ar[ur]_{\rho'}	&	
}
\]
Hence, if $\xymatrix{\GG & \ar[l]_{\lambda} \mathbf{K} \ar[r]^{\rho} & \mathbf{H}}$ is a Morita equivalence of quasi-presymplectic groupoids $(\GG,\omega\oplus \eta)$ and $(\mathbf{H},\nu\oplus\chi)$, then so is $\xymatrix{\GG & \ar[l]_{\lambda'} \mathbf{K'} \ar[r]^{\rho'} & \mathbf{H}}$.  In particular, if $(\GG,\omega\oplus \eta)$ and $(\mathbf{H},\nu\oplus\chi)$ are Morita equivalent quasi-presymplectic groupoids, there exists a Morita equivalence (of quasi-presymplectic groupoids) $\xymatrix{\GG & \ar[l]_{\lambda} \mathbf{K} \ar[r]^{\rho} & \mathbf{H}}$, where 
the object maps $\lambda_0:K_0 \to G_0$ and $\rho_0:K_0 \to H_0$ are surjective submersions (e.g.~see \cite[Remark 1.38]{li2015higher}).

\subsection{Related Hamiltonian spaces}

Recall from \cite[Theorem 4.19]{xu2004momentum}, that Morita equivalent quasi-presymplectic groupoids $(\GG,\omega\oplus \eta)$ and $(\mathbf{H},\nu\oplus\chi)$ give rise to equivalent theories of Hamiltonian actions.  That is, given a Hamiltonian $\GG$-space $(X,\omega_X,\Phi)$ one may construct a corresponding Hamiltonian $\mathbf{H}$-space $(M,\omega_M,\mu)$ and vice versa.  In Theorem \ref{t:moritapreq} below, we will show that prequantization respects this correspondence.  

To begin, we recall some aspects of the correspondence in \cite{xu2004momentum}, in a special case which will suffice in our setting.  (We refer the reader to \cite[Proposition 4.23]{xu2004momentum} for details.)
Let $(\GG,\omega\oplus \eta)$ and $(\mathbf{H},\nu\oplus\chi)$ be proper quasi-presymplectic groupoids, and let $\mathbf{F}:\GG \to \mathbf{H}$ be an equivalence, viewed as an equivalence of quasi-presymplectic groupoids (i.e.~with $\lambda = \id$, $\rho=\mathbf{F}$). Suppose in addition that $F_0:G_0 \to H_0$ is a surjective submersion.  Let $\alpha \in \Omega^2(G_0)$ satisfy 
\begin{equation*}
\mathbf{F}^*(\nu \oplus \chi) -\omega\oplus \eta = \delta \alpha
\end{equation*}
as in Remark \ref{r:firstkind}.

Let $(M,\omega_{M},\mu)$ be a pre-Hamiltonian $\mathbf{H}$-space.  Then the corresponding $\GG$-space $(X,\omega_X,\Phi)$ is obtained by setting 
\[
X = G_0 \times_{H_0} M,
\]
with natural $\GG$-action, 
$$
g\cdot (x,m) = (t(g),F_1(g)\cdot m)
$$ for $g\in G_1$ satisfying $s(g)=x$,
and moment map $\Phi:X \to G_0$ given by projection to the first factor.

Conversely, $(M,\omega_{M},\mu)$ is obtained from $(X,\omega_X,\Phi)$ by setting $M$ to be the quotient 
\[
M = X/G_\star
\]
where  $G_\star \toto G_0$ is the subgroupoid  $G_\star \subset G_1$ consisting of all arrows $g \in G_1$ such that $F_1(g)\in H_1$ is an identity map (i.e.~in the image of the unit $\epsilon_{\mathbf{H}}$).  The moment map $\mu:M \to H_0$ is induced from the moment map $\Phi$, $\mu([p]) = F_0(\Phi(p))$.  The $\mathbf{H}$-action can be described as follows. For a $G_\star$-orbit $[p]$ and $h\in H_1$ with $s(h)=\mu([p])$, choose $y\in G_0$ with $F_0(y)=t(h)$. Then there exists a unique $g\in G_1$ with $F_1(g)=h$, and we set $h\cdot [p]=[g\cdot p]$.

The corresponding 2-forms on $X$ and $M$ are determined by the relation:
\begin{equation} \label{eq:2forms}
f_0^*\omega_{M} = \omega_X -  \Phi^*\alpha,
\end{equation}
where $f_0:X\to M$ denotes the quotient map.

\begin{lemma} \label{l:morita}
Let $\GG$ and $\mathbf{H}$ be proper Lie groupoids, and let $\mathbf{F}:\GG \to \mathbf{H}$ be an equivalence with $F_0$ a surjective submersion.  Let $X$ and $M$ be  $\GG$ and $\mathbf{H}$-spaces with moment maps $\Phi$ and $\mu$, respectively, obtained from either of the above constructions.  Then the natural morphism $\GG\ltimes X \to \mathbf{H} \ltimes M$ is an equivalence. 
\end{lemma}
\begin{proof}
It is straightforward to verify that the above constructions for $X$ and $M$ are inverse to each other (up to diffeomorphism).  
Therefore, we may assume $X= G_0 \times_{H_0} M$.
Abusing notation, let $\mathbf{f}:\GG\ltimes X \to \mathbf{H}\ltimes M$ be the induced map of action groupoids:
\[
\xymatrix{
G_1\times_{G_0}  G_0 \times_{H_0} M \ar@<-.7ex>[d] \ar@<+.7ex>[d] \ar[r]^-{f_1}  & H_1\times_{H_0}M \ar@<-.7ex>[d] \ar@<+.7ex>[d]\\
 G_0 \times_{H_0} M \ar[r]^-{f_0} & M
}
\]
Since $F_0$ is a surjective submersion, so is the pullback $f_0$, and hence so is the map
\[
t\circ \pr_1: (H_1\times_{H_0}M) \times_{M} (G_0 \times_{H_0} M) \to M.
\]
It remains to verify 
\begin{equation} \label{eq:ff}
G_1\times_{G_0}  G_0 \times_{H_0} M  \cong (( G_0 \times_{H_0} M) \times ( G_0 \times_{H_0} M )) \times_{M\times M}  (H_1\times_{H_0}M)
\end{equation}
To that end, let $(x_1,m_1; x_2, m_2; h,m)$ be an element from the right hand side of \eqref{eq:ff}, so that $\phi(x_i)=\mu(m_i)$ ($i=1,2$), $m_1=m$, and $m_2=h\cdot m$.  Since $\mathbf{F}$ is an equivalence, there is a unique $g\in G_1$ with $s(g)=x_1$, $t(g)=x_2$ and $(g,x_1,m_1)$ defines an arrow on the left side of \eqref{eq:ff}. This gives the required identification.
\end{proof}

Now suppose that $\xymatrix{\GG & \ar[l]_{\lambda} \mathbf{K} \ar[r]^{\rho} & \mathbf{H}}$ is a Morita equivalence of the proper quasi-symplectic groupoids $(\GG,\omega\oplus \eta)$ and $(\mathbf{H},\nu\oplus\chi)$, where $\lambda_0$ and $\rho_0$ are surjective submersions.  Let $\alpha\in \Omega^2(K_0)$ satisfy \eqref{eq:3forms}.

We may iterate the construction above to associate a Hamiltonian $\mathbf{H}$-space to a Hamiltonian $\GG$-space \emph{via} the corresponding $\mathbf{K}$-space.  Indeed, let $(X,\omega_X,\Phi)$ be a pre-Hamiltonian $\GG$-space.  Using the above construction, we may form a Hamiltonian $\mathbf{K}$-space $(Z,\omega_Z,\Psi)$ for the quasi-presymplectic groupoid $(\mathbf{K},\rho^*(\nu\oplus \chi))$.  It is straightforward to verify that $(Z,\omega_Z+\Psi^*\alpha,\Psi)$ is a Hamiltonian $\mathbf{K}$-space for the quasi-presymplectic groupoid $(\mathbf{K},\lambda^*(\omega\oplus \eta))$.  Hence we may form the Hamiltonian $\mathbf{H}$-space $(M,\omega_M,\mu)$.  Using \eqref{eq:2forms}, it follows that the 2-forms $\omega_X$, $\omega_M$, and $\omega_Z$ satisfy:
\[
\ell^*\omega_X = \omega_Z + \Psi^*\alpha \quad \text{and} \quad r^*\omega_M=\omega_Z,
\]
where $\ell:Z\to X$ and $r:Z\to M$ denote the natural quotient maps arising in the construction.

A straightforward application of  Lemma \ref{l:morita} gives the following Proposition.

\begin{proposition} \label{p:morita}
Let $\GG$ and $\mathbf{H}$ be   proper Lie groupoids and suppose $\xymatrix{\GG & \ar[l]_{\lambda} \mathbf{K} \ar[r]^{\rho} & \mathbf{H}}$ is a Morita equivalence. Let $X$, $Z$ and $M$ be  $\GG$, $\mathbf{K}$ and $\mathbf{H}$-spaces with moment maps $\Phi$, $\Psi$, and $\mu$, respectively,  as in the above construction.  Then  the natural morphisms $\xymatrix{\GG \ltimes X& \ar[l] \mathbf{K}\ltimes Z \ar[r] & \mathbf{H}\ltimes M}$ provide a Morita  equivalence of the corresponding action groupoids. 
\end{proposition}

We are now ready to establish the main result of this section, showing that prequantization respects related Hamiltonian spaces.

\begin{theorem} \label{t:moritapreq}
Let $(\GG,\omega\oplus \eta)$ and $(\mathbf{H},\nu\oplus\chi)$ be proper quasi-presymplectic groupoids, and let $\xymatrix{\GG & \ar[l]_{\lambda} \mathbf{K} \ar[r]^{\rho} & \mathbf{H}}$ be an equivalence of quasi-presymplectic groupoids where $\lambda_0:K_0 \to G_0$ and $\rho_0:K_0 \to H_0$ are surjective submersions.  Let $(X,\omega_{X},\Phi)$ and $(M,\omega_{M},\mu)$ be Hamiltonian $\GG$ and $\mathbf{H}$-spaces, respectively, as in the discussion above.
 Then  $(X,\omega_{X},\Phi)$ admits a prequantization if and only if  $(M,\omega_{M},\mu)$ does.
\end{theorem}
\begin{proof}
Let $Z$ be the intermediate Hamiltonian $\mathbf{K}$-space with moment map $\Psi$, as in the discussion preceding Proposition \ref{p:morita}.
By Proposition \ref{p:morita}, the following diagram of Lie groupoids commutes,
\[
\xymatrix{
\mathbf{H} \ltimes \ar[r]^-{\boldsymbol{\mu}}  Z & \mathbf{K}  \\
\mathbf{K} \ltimes \ar[r]^-{\boldsymbol{\Psi}} \ar[u]^-{r} \ar[d]_{\ell} Z & \mathbf{K} \ar[u]_-{\rho} \ar[d]^{\lambda} \\
\mathbf{G} \ltimes X \ar[r]^-{\PPhi} & \mathbf{G}
}
\]
where the vertical maps are equivalences.

Being equivalences, the vertical maps induce isomorphisms on cohomology $H^*(-)$ (with any coefficients).  The two commutative squares in the above diagram induces  maps on relative cohomology (with any coefficients)  $H^3(\boldsymbol{\mu}) \rightarrow H^3(\boldsymbol{\Psi}) \leftarrow H^3(\PPhi)$, which, by easy applications of the 5-Lemma, are isomorphisms.  
By Proposition \ref{p:integrality}, it suffices to check that $[(\omega_X,\omega\oplus \eta)] \in H^3(\PPhi;\RR)$ is integral if and only if $[(\omega_M,\nu\oplus \chi)] \in H^3(\boldsymbol{\mu};\RR)$ is integral.  Using the above pair of isomorphisms induced from the commutative diagram, it then suffices to check that
$(\ell,\lambda)^*[(\omega_X,\omega\oplus \eta)]=(r,\rho)^*[(\omega_M,\nu\oplus \chi)]$. Indeed, we have
\begin{align*}
(\ell,\lambda)^*[(\omega_X,\omega\oplus \eta)] & = (r^*\omega_M +\Psi^*\alpha,\rho^*(\nu\oplus \chi)-\delta \alpha) \\
& = (r^*\omega_M,\rho^*(\nu\oplus \chi)) + \delta(0,\alpha). 
\end{align*}
\end{proof}

\subsubsection*{Hamiltonian loop group actions and quasi-Hamiltonian spaces}

For the remainder of this section, let $G$ be a compact Lie group with bi-invariant inner product $\langle \cdot, \cdot \rangle$ on $\g$, and $LG=\operatorname{Map}(S^1,G)$ as in Example \ref{eg:loopgroup}.

As an application of Theorem \ref{t:moritapreq} to the equivalence (see \cite[Theorem 8.3]{alekseev1998lie} and \cite[Corollary 4.28]{xu2004momentum}) between Hamiltonian $LG$-actions with proper moment map and quasi-Hamiltonian $G$-actions with group-valued moment map, we obtain the following corollary (\emph{cf.} \cite[Theorem A.7]{krepski2008pre} which assumes $G$ is simply connected).

\begin{corollary} \label{c:ammloop}
Let $(X,\omega_X,\Phi)$ be a pre-Hamiltonian $LG$-space, and $(M,\omega_M,\mu)$ its corresponding quasi-Hamiltonian $G$-space.  Then $(X,\omega_X,\Phi)$ admits a prequantization if and only if $(M,\omega_M,\mu)$ does.
\end{corollary}

A necessary condition (and ingredient) for the existence of a prequantization is the existence of $\GG$-equivariant DD-bundle representing the quasi-presymplectic structure $\omega\oplus \eta$.  For the case of the symplectic groupoid $LG\times L\g^* \toto L\g^*$ this amounts to a central $S^1$-extension (by Theorem \ref{t:compare} (1)).  And by \cite[Theorem 3.3]{behrend2003equivariant}, this central $S^1$-extension of groupoids corresponds to a central extension of Lie groups,
\[
1\to S^1 \to \widehat{LG} \to LG \to 1.
\]
For simple $G$, such central extensions are classified (see \cite{presley1986loop} and \cite{laredo1999positive}), and a necessary condition for the existence of such a central extension is that the inner product $\langle\cdot,\cdot \rangle$ be a multiple of $l_b\cdot B$, where $B$ denotes the \emph{basic inner product} on $\g$, and  $l_b$ is the \emph{basic level}.  Recall $B$ is the invariant inner product on $\g$ normalized to make short co-roots have squared length $2$, and the integer $l_b$ is the smallest integer $l$ such that $l\cdot B(\lambda_1,\lambda_2)\in \ZZ$ for all elements $\lambda_1$, $\lambda_2$ of the integral lattice $\Lambda \subset \mathfrak{t}$  ($\mathfrak{t}$ is the Lie algebra of a maximal torus $T\subset G$).  The integers $l_b$ are computed in \cite{laredo1999positive} for each compact simple Lie group $G$ (e.g.~if $G$ is simply connected, $l_b=1$). Therefore, by Corollary \ref{c:ammloop}, we obtain:

\begin{corollary}
Let $G$ be a compact simple Lie group and let $l\cdot B$ be a positive integer multiple of the basic inner product on $\g$. If the quasi-Hamiltonian $G$-space $(M,\omega_M,\mu)$ admits a prequantization, then $l$ is a multiple of the basic level $l_b$.
\end{corollary}

\section{Symplectic quotients}\label{s:red}

In this section, show how the definition of prequantization \ref{d:preq} is compatible with symplectic quotients.

Let $(\GG,\omega\oplus \eta)$ be a proper quasi-presymplectic groupoid, and let $(X,\omega_X,\Phi)$ be a pre-Hamiltonian $\GG$-space.  Suppose that $z\in G_0$ is a regular value for $\Phi:X\to G_0$.  Then the $\GG$-action on $X$ restricts to a $G_1(z)$-action on the level set manifold $\Phi^{-1}(z)\subset X$, where $G_1(z)=\{g\in G_1 \, \vert \, s(g)=z=t(g)\}$ denotes the isotropy group of $z$.  Below, we show how a prequantization for $(X,\omega_X,\Phi)$ pulls back to a $G_1(z)$-equivariant $S^1$-bundle over $\Phi^{-1}(z)$.  When the $G_1(z)$-action on the level set is free, we may view this as a prequantization of the resulting symplectic quotient $\Phi^{-1}(z)/G_1(z)$ (see \cite[Theorem 3.18]{xu2004momentum}).

\begin{theorem} \label{t:red}
Let $(\GG,\omega\oplus \eta)$ be a proper quasi-presymplectic groupoid, and let $(X,\omega_X,\PPhi)$ be a pre-Hamiltonian $\GG$-space.  Suppose that $z\in G_0$ is a regular value for $\Phi:X\to G_0$ and let $j:\Phi^{-1}(z) \hookrightarrow X$ denote the inclusion map.  A prequantization of $(X,\omega_X,\PPhi)$ gives rise to a $G_1(z)$-equivariant $S^1$-bundle over $\Phi^{-1}(z)$ whose real $G_1(z)$-equivariant curvature class is $[j^*\omega_X] \in H^2(G_1(z)\ltimes \Phi^{-1}(z);\RR)$.
\end{theorem}
\begin{proof}
As in the proof of Theorem \ref{t:compare} (2), we shall use the framework of equivariant differential characters.  Let $\zeta=(\zeta_0, \zeta_1, [\zeta_2])=((c_0,h_0,\eta),(c_1,h_1,\omega),[c_2,h_2,0])$ be an equivariant differential character (of degree 3) and let Let  $((B_0,F_0,\omega_X),[B_1,F_1,0]):0\to \PPhi^*\zeta $ be a prequantization.  
Since the composition $\Phi\circ j$ is constant,  then $j^*((B_0,F_0,\omega_X),[B_1,F_1,0])$ defines an equivariant differential character of degree 2.  Indeed, a straightforward verification shows $(j^*B_0, j^*F_0, j^*\omega_X)$ is cocycle in $DC_1^2(\Phi^{-1}(z))$ and that $[j^*B_1,j^*F_1,0]$ defines the required 1-isomorphism $\partial_0^*(j^*B_0, j^*F_0, j^*\omega_X) \to \partial_1^*(j^*B_0, j^*F_0, j^*\omega_X)$ to give an equivariant differential character of degree 2.  It follows that $[j^*\omega_X]$ in $H^2(G_1(z)\ltimes \Phi^{-1}(z);\RR)$ is the resulting $G_1(z)$-equivariant curvature class.
\end{proof}

\begin{remark}  
The proof of Theorem \ref{t:red} shows how one can describe the resulting $G_1(z)$-equivariant $S^1$-bundle with curvature class $j^*\omega$ in bundle-theoretic terms.  On the level set $\Phi^{-1}(z)$, there are two Morita trivializations of the equivariant $DD$-bundle $j^*\PPhi^*(\calA,\calE,\tau)$: one resulting from the pullback of the prequantization $(\calF,\alpha):(\CC,\CC,\id) \dashto \PPhi^*(\calA,\calE,\tau)$ along the inclusion $j$, and the canonical trivialization coming  from the fact that the composition $\Phi \circ j$ is trivial.  Their difference defines a line bundle (or equivalently its associated $S^1$-bundle) over $\Phi^{-1}(z)$.
\end{remark}

\begin{remark} \label{r:quot}
Under the conditions of Theorem \ref{t:red}, if we assume additionally that $G_1(z)$ acts freely on the level set $\Phi^{-1}(z)$, we see that the $G_1(z)$-equivariant $S^1$-bundle obtained in the Theorem descends to a prequantum $S^1$-bundle on the symplectic quotient $\Phi^{-1}(z)/G_1(z)$.
\end{remark}


\bibliographystyle{plain}

\end{document}